\documentclass{amsart}
\usepackage{pgfplots}
\usepackage{amsmath}
\pgfplotsset{compat=1.18}
\usepackage{sansmath} % optional: math font consistent with sans serif text
\sansmath
\usepackage{graphicx}
\usepackage{slashed}
\usepackage{enumerate}
\usepackage{dsfont}
\usepackage{mathtools,cancel}
\usepackage[bottom]{footmisc}
\usepackage[scr=rsfso]{mathalfa}
\usepackage{tikz-cd}
\usepackage[hidelinks]{hyperref}
{

}
\usepackage{amsfonts}
\usepackage{psfrag}
\usepackage{color}
\usepackage{mathtools}
\usepackage{pgf,tikz}
\usepackage{mathrsfs}
\usetikzlibrary{arrows}
\usepackage{fancyhdr}
\usepackage{amssymb}
\usepackage{slashed}
\usepackage{enumitem}

\usepackage[T1]{fontenc}
\usepackage[utf8]{inputenc}
\newtheorem{definition}{Definition}
\newtheorem{remark}{Remark}

\newtheorem{proposition}{Proposition}[section]

\newtheorem{theorem}{Theorem}[section]
\newtheorem{corollary}{Corollary}[section]
\newtheorem{lemma}{Lemma}[section]

\newcommand{\be}{\begin{equation}}
\newcommand{\ee}{\end{equation}}

\newcommand{\bm}{\begin{align}*}
\newcommand{\enm}{\end{align}*}

\newcommand{\bespeq}{\begin{equation}\begin{split}}
\newcommand{\espeq}{\end{split}\end{equation}}

\renewcommand{\div}{\mbox{div }}

\newcommand{\tr}{\mbox{tr}}

\newcommand\restri[2]{{% we make the whole thing an ordinary symbol
		\left.\kern-\nulldelimiterspace % automatically resize the bar with \right
		#1 % the function
		\right|_{#2} % this is the delimiter
}}

\definecolor{ffqqqq}{rgb}{1.,0.,0.}
\definecolor{uuuuuu}{rgb}{0.26666666666666666,0.26666666666666666,0.26666666666666666}

\makeatletter
\def\ps@pprintTitle{%
  \let\@oddhead\@empty
  \let\@evenhead\@empty
  \let\@oddfoot\@empty
  \let\@evenfoot\@oddfoot
}
\def\@author#1{\g@addto@macro\elsauthors{\normalsize%
    \def\baselinestretch{1}%
    \upshape\authorsep#1\unskip\textsuperscript{%
      \ifx\@fnmark\@empty\else\unskip\sep\@fnmark\let\sep=,\fi
      \ifx\@corref\@empty\else\unskip\sep\@corref\let\sep=,\fi
      }%
    \def\authorsep{\unskip,\space}%
    \global\let\@fnmark\@empty
    \global\let\@corref\@empty  %% Added
    \global\let\sep\@empty}%
    \@eadauthor={#1}
}
\makeatother
\begin{document}

\title{Compact Proof of the Positivity of Quasi-Local Masses for a class of Initial Data}
\author{Puskar Mondal}%Beijing Institute of Mathematical Sciences and Applications, Yau Mathematical Sciences Center}
\author{Shing-Tung Yau}

\maketitle

\begin{abstract}
\noindent We prove a purely quasi-local positivity theorem for the Wang--Yau mass for a class of initial data whose Jang deformation, after a boundary-preserving conformal reduction to zero scalar curvature, lies in a sufficiently small transverse--traceless (TT) perturbative neighborhood of a strictly convex Euclidean fill-in.We explicitly construct a nontrivial class of physical initial data whose admissible Jang reductions realize this TT-generated sector. The argument reduces the Wang--Yau energy to the Brown--York mass of the resulting scalar-flat compact metric, together with nonnegative bulk terms determined by the Jang deformation, and establishes strict positivity by computing the second variation of the Brown--York functional at the Euclidean metric in transverse--traceless directions. The proof is entirely confined to the compact fill-in and uses neither an asymptotically flat extension nor the positive mass theorem. This gives a partial answer to a question of R. Schoen concerning a genuinely quasi-local proof of positivity for quasi-local mass.

\end{abstract}

\setcounter{tocdepth}{2}
{\hypersetup{linkcolor=black}
\small
\tableofcontents
}

\section{Introduction}
\noindent We begin by fixing the spacetime conventions used throughout the paper.  Let
\(({\mathcal M}^{3+1},\widehat g)\) be a smooth, time-oriented, globally
hyperbolic Lorentzian four-manifold with signature \((-+++)\), and let
\(T_{\mu\nu}\) denote the stress-energy tensor of the matter fields.  We work
in geometrized units \(8\pi G=1\), so that the Einstein equations are
\begin{eqnarray}
        \widehat R_{\mu\nu}
        -
        {1\over 2}\widehat R\,\widehat g_{\mu\nu}
        =
        T_{\mu\nu}.
\end{eqnarray}
Here \(\widehat R_{\mu\nu}\) and \(\widehat R\) denote respectively the Ricci curvature
tensor and scalar curvature of the spacetime metric \(\widehat g\).

\noindent The Wang--Yau quasi-local mass is associated with a closed spacelike
two-surface \(\Sigma\subset {\mathcal M}\), usually realized as the boundary
of a compact spacelike hypersurface.  Thus let
\(\Omega\subset {\mathcal M}\) be a smooth compact spacelike hypersurface,
diffeomorphic to a three-ball, with boundary
\begin{eqnarray}
        \partial\Omega=\Sigma .
\end{eqnarray}
Let \(n\) be the future-directed timelike unit normal to \(\Omega\).  In a
local \(3+1\) decomposition adapted to the hypersurface \(\Omega\), one may
write
\begin{eqnarray}
        n=N^{-1}(\partial_t-X),
\end{eqnarray}
where \(N>0\) is the lapse function and \(X=X^i\partial_i\) is the shift vector
field tangent to the \(t=\mathrm{constant}\) slices.  In these coordinates, the
spacetime metric takes the ADM form
\begin{eqnarray}
        \widehat g
        =
        -N^2dt\otimes dt
        +
        g_{ij}(dx^i+X^i dt)\otimes(dx^j+X^j dt).
\end{eqnarray}
The Riemannian metric induced on \(\Omega\) is
\begin{eqnarray}
        g_{ij}:=\widehat g(\partial_i,\partial_j).
\end{eqnarray}

\noindent With the sign convention
\begin{eqnarray}
        K_{ij}
        :=
        \widehat g(\widehat\nabla_{\partial_i}n,\partial_j),
\end{eqnarray}
we denote by \(K\) the second fundamental form of \(\Omega\) in
\(({\mathcal M},\widehat g)\).  The Einstein equations imply the constraint
equations on the initial data set \((\Omega,g,K)\):
\begin{eqnarray}
        R(g)-K_{ij}K^{ij}+(\mathrm{tr}_g K)^2
        =
        2\mu,
        \label{eq:hamiltonian-constraint}
\end{eqnarray}
and
\begin{eqnarray}
        \nabla^jK_{ij}-\nabla_i(\mathrm{tr}_g K)
        =
        J_i .
        \label{eq:momentum-constraint}
\end{eqnarray}
Here \(R(g)\), \(\nabla\), \(|\cdot|_g\), and \(\mathrm{tr}_g\) are computed
with respect to \(g\), while
\begin{eqnarray}
        \mu:=T(n,n),
        \qquad
        J_i:=T(n,\partial_i)
\end{eqnarray}
are respectively the energy density and momentum density measured by the
observers normal to \(\Omega\).

\noindent We assume the dominant energy condition.  In invariant spacetime form, this
means that for every future-directed causal vector \(V\), the vector
\begin{eqnarray}
        -T^\mu{}_\nu V^\nu
\end{eqnarray}
is future-directed causal.  Equivalently, in the \(3+1\) decomposition, the
energy and momentum densities satisfy the pointwise inequality
\begin{eqnarray}
        \mu\geq |J|_g,
        \qquad
        |J|_g:=\sqrt{g^{ij}J_iJ_j}.
        \label{eq:dominant-energy-initial-data}
\end{eqnarray}
Thus the local energy-momentum density \((\mu,J)\) is non-spacelike and
future-directed.  This condition expresses, in geometric form, the causal
propagation of matter-energy.

\noindent A fundamental special case is the time-symmetric regime,
\begin{eqnarray}
        K\equiv 0.
\end{eqnarray}
Then the momentum constraint is identically satisfied, and the Hamiltonian
constraint reduces to
\begin{eqnarray}
        R(g)=2\mu .
\end{eqnarray}
Consequently, under the dominant energy condition,
\begin{eqnarray}
        R(g)\geq 0.
\end{eqnarray}
Thus, in the time-symmetric setting, the Lorentzian dominant energy condition
reduces precisely to the nonnegativity of the scalar curvature of the
Riemannian initial data set \((\Omega,g)\).

\noindent Let \((\Sigma,\sigma)\) be a smooth closed spacelike two-surface in the
physical spacetime \(({\mathcal M},\widehat g)\), with the topology of
\(\mathbb S^2\), and suppose that \(\Sigma\) bounds a smooth compact spacelike
hypersurface \(\Omega\).  Thus
\[
        \partial\Omega=\Sigma,
\]
and \(\sigma\) denotes the Riemannian metric induced on \(\Sigma\) by
\(\widehat g\), equivalently by the induced metric \(g\) on \(\Omega\).  The
spacetime is assumed to satisfy the regularity, time-orientation, global
hyperbolicity, and energy assumptions fixed above.

\noindent The basic quasi-local problem is to assign a precise geometric meaning to the
question:
\[
        \hbox{what is the total gravitational energy contained in }
        \Omega \hbox{ with boundary } \Sigma ?
\]
A satisfactory definition of quasi-local energy should satisfy, at minimum,
the following structural requirements.  First, it should be nonnegative under
an appropriate energy condition on the spacetime, typically the dominant energy
condition.  Second, it should satisfy a rigidity property: it should vanish for
a surface enclosing a domain in the reference flat spacetime, at least in the
appropriate reference configuration.  Third, since gravitational energy has no
local scalar density in general relativity, the definition should be genuinely
geometric and should detect both the contribution of matter fields through
the stress-energy tensor and the contribution of the free gravitational field,
encoded in the curvature of the ambient spacetime.

\noindent The Brown--York and Liu--Yau constructions constitute two fundamental
precursors to the Wang--Yau theory.  Motivated by the Hamilton--Jacobi analysis
of the gravitational action, Brown and York
\cite{brown1992quasilocal,brown1993quasilocal} defined a quasi-local energy by
comparing the mean curvature of the physical surface with the mean curvature
of an isometric embedding of the same intrinsic two-geometry into Euclidean
space.  The Liu--Yau mass
\cite{liu2003positivity,liu2006positivity} modifies this construction by using
the spacetime mean-curvature vector.  Both constructions rely on the classical
isometric embedding theory for metrics of positive Gaussian curvature on
\(\mathbb S^2\), in particular the Weyl--Pogorelov theorem
\cite{pogorelov1952regularity}, which guarantees the existence and uniqueness,
up to rigid motions, of an isometric embedding into \(\mathbb R^3\) (Note \cite{pyau2} for a notion of mass including gauge fields).

\begin{definition}[Brown--York mass]
Let \(\Sigma\) be a smooth spacelike two-sphere in the spacetime
\(({\mathcal M},\widehat g)\), and suppose that \(\Sigma=\partial\Omega\) for
a smooth compact spacelike hypersurface \((\Omega,g)\).  Let
\[
        \sigma=g|_{\Sigma}
\]
be the induced metric on \(\Sigma\), and assume that its Gaussian curvature is
strictly positive:
\[
        K_{\sigma}>0.
\]
By the Weyl--Pogorelov isometric embedding theorem, there exists an isometric
embedding
\[
        i:(\Sigma,\sigma)\longrightarrow (\mathbb R^3,\delta),
\]
unique up to Euclidean rigid motions.  Let \(k_0\) denote the mean curvature of
the embedded surface \(i(\Sigma)\subset\mathbb R^3\), computed with respect to
the outward unit normal in Euclidean space.  Let \(k\) denote the mean
curvature of \(\Sigma\subset\Omega\), computed with respect to the outward
unit normal in the physical spacelike hypersurface \((\Omega,g)\).  Then the
Brown--York mass of \(\Sigma\) relative to \(\Omega\) is
\begin{eqnarray}
        m_{\mathrm{BY}}(\Sigma,\Omega)
        :=
        {1\over 8\pi}
        \int_{\Sigma}(k_0-k)\,d\mu_{\sigma}.
        \label{eq:brown-york-mass}
\end{eqnarray}
Here \(d\mu_{\sigma}\) is the area form of the induced metric \(\sigma\).
\end{definition}
The Liu-Yau mass is defined similarly with mean curvature $k$ replaced by the Lorentzian norm of the spacetime mean curvature of $\Sigma$ in $(M,\widehat{g})$ i.e., 
\begin{definition}[Li-Yau Mass]
Assume that the topological $2-$sphere $( \Sigma,\sigma)$ bounds a topogical ball $\Omega$ in the spacetime $(M,\widehat{g})$. Also assume that the Gaussian curvature of $( \Sigma, \sigma)$ is strictly positive. The Liu-Yau mass is defined as follows
 \begin{eqnarray}
 M^{LY}:=\frac{1}{8\pi}\int_{\Sigma}(k_{0}-|H|)\mu_{\sigma},   
\end{eqnarray}   
where $|H|$ is the Lorentzian norm of the mean curvature vector of $(\Sigma,\sigma)$ in the spacetime $(M,\widehat{g})$.
\end{definition}

\noindent The Brown--York and Liu--Yau constructions satisfy natural positivity
properties under their respective geometric hypotheses; nevertheless,
this positivity is, in an important sense, excessively strong. As
observed in \cite{murchadha2004comment}, there exist spacelike
$2$-surfaces embedded in Minkowski spacetime for which both the
Brown--York and Liu--Yau masses are strictly positive. Such examples
demonstrate that these quantities do not, by themselves, vanish on all
surfaces contained in the reference spacetime and therefore do not fully
encode the cancellation between intrinsic geometry and spacetime
momentum that one expects from a physically normalized quasi-local
energy. This deficiency is closely related to the fact that the
reference term in these definitions is determined essentially by the
intrinsic geometry of the surface and the spacelike hypersurface it is embedded in, while the full normal-bundle and
momentum information of the embedding in spacetime is not incorporated
in a sufficiently rigid manner.

Motivated by this issue, Wang and Yau introduced in
\cite{yau,yau1} a quasi-local energy in which the reference
configuration is specified by an isometric embedding of the surface into
Minkowski spacetime together with a choice of observer, and in which the
physical and reference Hamiltonian data are compared after optimizing
over the corresponding time function. The resulting Wang--Yau mass
satisfies the expected rigidity property for surfaces in Minkowski
spacetime and provides a substantially sharper incorporation of the
spacetime momentum degrees of freedom. We also note that Alaee, Khuri,
and Yau have recently proposed a different quasi-local mass
construction for surfaces satisfying certain topological hypotheses; see
\cite{alaee2023quasi}. In what follows, we recall the Wang--Yau
definition in the form adapted to our subsequent analysis.

\noindent Let $\Sigma$ be a spacelike topological $2$--sphere in the physical
spacetime $(M,\widehat g)$, and assume throughout that its mean
curvature vector $\mathbf H$ is spacelike. Associated with $\mathbf H$
is the future-directed timelike normal vector $\mathbf J$, obtained by
Lorentzian reflection of $\mathbf H$ across the future outgoing null
direction in the normal bundle of $\Sigma$. Equivalently, $\mathbf J$
is characterized, up to the prescribed time orientation, by
\[
\widehat g(\mathbf H,\mathbf J)=0,
\qquad
\widehat g(\mathbf J,\mathbf J)
=
-\widehat g(\mathbf H,\mathbf H)
=
-|\mathbf H|_{\widehat g}^{2},
\]
with $\mathbf J$ future directed.

\noindent The geometric data entering the Wang--Yau construction are the triple
\[
\left(
\sigma,\,
|\mathbf H|_{\widehat g},\,
\alpha_{\mathbf H}
\right),
\]
where $\sigma$ is the metric induced on $\Sigma$ by $\widehat g$,
$|\mathbf H|_{\widehat g}$ is the Lorentzian norm of the spacelike mean
curvature vector, and $\alpha_{\mathbf H}$ is the connection
$1$--form of the normal bundle in the mean-curvature gauge. More
precisely,
\begin{equation}
\label{eq:connection}
\alpha_{\mathbf H}(Y)
:=
\widehat g
\left(
\widehat\nabla_Y
\frac{\mathbf J}{|\mathbf H|_{\widehat g}},
\frac{\mathbf H}{|\mathbf H|_{\widehat g}}
\right),
\qquad
Y\in T\Sigma .
\end{equation}

\noindent Choose an orthonormal frame $\{e_3,e_4\}$ of the normal bundle
$T^\perp\Sigma$ satisfying
\[
\widehat g(e_3,e_3)=1,
\qquad
\widehat g(e_4,e_4)=-1,
\qquad
\widehat g(e_3,e_4)=0,
\]
where $e_4$ is taken to be future-directed.  Thus, relative to any
future-directed timelike vector field $\partial_t$ compatible with the
chosen time orientation,
\[
\widehat g(e_4,\partial_t)<0.
\]

\noindent We next choose an isometric embedding
\[
i_0:(\Sigma,\sigma)\longrightarrow
(\mathbb R^{1,3},\eta).
\]
In local coordinates $\{x^a\}_{a=1}^{2}$ on $\Sigma$, writing
\[
i_0(x)=X^\mu(x)\frac{\partial}{\partial X^\mu},
\]
the isometric embedding condition is
\begin{equation}
\label{eq:isometric-embedding-Minkowski}
\sigma_{ab}
=
\eta_{\mu\nu}
\frac{\partial X^\mu}{\partial x^a}
\frac{\partial X^\nu}{\partial x^b}.
\end{equation}
Let $\{e_{30},e_{40}\}$ be an orthonormal frame of the normal bundle
of $i_0(\Sigma)$ in Minkowski spacetime, chosen so that
\[
\eta(e_{30},e_{30})=1,
\qquad
\eta(e_{40},e_{40})=-1,
\qquad
\eta(e_{30},e_{40})=0,
\]
with $e_{40}$ future directed.  The physical and reference normal
frames are understood to be chosen in the Wang--Yau canonical gauges
associated with the prescribed time function.

\noindent Let $T_0$ be a fixed future-directed unit timelike vector in
$\mathbb R^{1,3}$.  The time function associated with the reference
embedding is
\begin{equation}
\label{eq:tau-reference}
\tau
:=
-\eta(i_0,T_0).
\end{equation}
In the standard inertial frame, $T_0=\partial_t$, so that $\tau$ is
precisely the restriction of the Minkowski time coordinate to
$i_0(\Sigma)$.

\noindent We now recall the Wang--Yau quasi-local energy in the form that will be
used below.

\begin{definition}[Wang--Yau quasi-local energy]
\label{wy14}

Let $(M,\widehat g)$ be a time-oriented Lorentzian spacetime satisfying
the dominant energy condition, and let
$\Sigma\subset M$ be an embedded spacelike topological $2$--sphere
whose mean curvature vector $\mathbf H$ is spacelike.  Let
\[
i_0:(\Sigma,\sigma)\longrightarrow\mathbb R^{1,3}
\]
be an isometric embedding, let $T_0$ be a constant future-directed unit
timelike vector, and let
\[
\tau=-\eta(i_0,T_0)
\]
be the associated time function.  Denote by $\mathbf H_0$ the mean
curvature vector of $i_0(\Sigma)\subset\mathbb R^{1,3}$.

\noindent For an admissible time function $\tau$ in the sense of
Definition \ref{admi}, the corresponding Wang--Yau energy functional is
given by
\begin{align}
8\pi E_{\rm WY}(\Sigma,\tau)
={}&
\int_{\Sigma}
\left(
-\sqrt{1+|\nabla\tau|_{\sigma}^{2}}\,
\eta(\mathbf H_0,e_{30})
-
\eta(
\nabla^{\eta}_{\nabla\tau}e_{30},
e_{40})
\right)
d\mu_{\sigma}
\nonumber\\
&-
\int_{\Sigma}
\left(
-\sqrt{1+|\nabla\tau|_{\sigma}^{2}}\,
\widehat g(\mathbf H,e_3)
-
\widehat g(
\widehat\nabla_{\nabla\tau}e_3,
e_4)
\right)
d\mu_{\sigma}.
\label{eq:wy-energy-functional}
\end{align}
The Wang--Yau quasi-local mass is then defined by minimizing over the
admissible class:
\begin{equation}
\label{eq:WY-mass-inf}
M_{\rm WY}(\Sigma)
:=
\inf_{\tau\in\mathcal A(\Sigma)}
E_{\rm WY}(\Sigma,\tau),
\end{equation}
where $\mathcal A(\Sigma)$ denotes the set of admissible time
functions.

\end{definition}

\noindent The first integral in \eqref{eq:wy-energy-functional} represents the
reference Hamiltonian contribution determined by the isometric
embedding into Minkowski spacetime, whereas the second integral is the
corresponding physical Hamiltonian contribution determined by the
embedding of $\Sigma$ in $(M,\widehat g)$.  In particular, the
connection terms retain the normal-bundle information and hence encode
the momentum degrees of freedom that are absent from purely
Riemannian quasi-local constructions.

\noindent For the purposes of the present work it is useful to separate the
minimization over $\tau$ from the positivity problem.  Accordingly, for
each fixed admissible time function $\tau$ we introduce the
\emph{reduced Wang--Yau energy}
\begin{align}
8\pi E_{WY}(\Sigma,\tau)
:={}&
\int_{\Sigma}
\left(
-\sqrt{1+|\nabla\tau|_{\sigma}^{2}}\,
\eta(\mathbf H_0,e_{30})
-
\eta(
\nabla^{\eta}_{\nabla\tau}e_{30},
e_{40})
\right)
d\mu_{\sigma}
\nonumber\\
&-
\int_{\Sigma}
\left(
-\sqrt{1+|\nabla\tau|_{\sigma}^{2}}\,
\widehat g(\mathbf H,e_3)
-
\widehat g(
\widehat\nabla_{\nabla\tau}e_3,
e_4)
\right)
d\mu_{\sigma}.
\label{eq:wy2}
\end{align}
Thus
\begin{equation}
\label{eq:WY-mass-reduced-inf}
M_{\rm WY}(\Sigma)
=
\inf_{\tau\in\mathcal A(\Sigma)}
E_{WY}(\Sigma,\tau).
\end{equation}
Our immediate objective will therefore be to establish
\begin{equation}
\label{eq:reduced-WY-positivity-goal}
E_{WY}(\Sigma,\tau)\geq0
\end{equation}
for the class of admissible time functions and perturbative initial
data considered below.  Once such an estimate is available uniformly
over the relevant admissible class, the non-negativity of the
Wang--Yau quasi-local mass follows directly from
\eqref{eq:WY-mass-reduced-inf}.

\noindent There remains a gauge freedom in the choice of orthonormal frame of the
Lorentzian normal bundle $T^\perp\Sigma$.  More precisely, let
$\{e_3,e_4\}$ be an oriented orthonormal frame satisfying
\[
\widehat g(e_3,e_3)=1,
\qquad
\widehat g(e_4,e_4)=-1,
\qquad
\widehat g(e_3,e_4)=0,
\]
with $e_4$ future directed.  Any other oriented orthonormal frame
$\{e_3',e_4'\}$ with $e_4'$ future directed is related to
$\{e_3,e_4\}$ by a unique proper orthochronous Lorentz boost
\begin{equation}
\label{eq:normal-boost}
\begin{split}
e_3'
&=
\cosh\varphi\,e_3+\sinh\varphi\,e_4,\\
e_4'
&=
\sinh\varphi\,e_3+\cosh\varphi\,e_4,
\end{split}
\end{equation}
for some smooth boost parameter
\[
\varphi\in C^\infty(\Sigma).
\]
The same freedom is present in the reference normal frame
$\{e_{30},e_{40}\}$ along the Minkowski embedding $i_0(\Sigma)$.

\noindent Consequently, before fixing the normal gauge, the Hamiltonian boundary
density appearing in the physical contribution to the Wang--Yau energy
depends on the choice of $\varphi$.  To make this dependence explicit,
set
\[
a:=\sqrt{1+|\nabla\tau|_\sigma^2}.
\]
For the boosted frame \eqref{eq:normal-boost}, the physical Hamiltonian
density is
\begin{equation}
\label{eq:physical-Hamiltonian-boost}
\mathcal H_{\rm phys}(\varphi)
:=
-a\,\widehat g(\mathbf H,e_3')
-
\widehat g
\left(
\widehat\nabla_{\nabla\tau}e_3',
e_4'
\right).
\end{equation}
The normal connection transforms according to
\begin{equation}
\label{eq:normal-connection-boost}
\widehat g
\left(
\widehat\nabla_Ye_3',
e_4'
\right)
=
\widehat g
\left(
\widehat\nabla_Ye_3,
e_4
\right)
-
Y(\varphi),
\qquad
Y\in T\Sigma,
\end{equation}
where the sign is determined by the convention
$\widehat g(e_4,e_4)=-1$.  Moreover,
\begin{equation}
\label{eq:H-boost}
\widehat g(\mathbf H,e_3')
=
\cosh\varphi\,
\widehat g(\mathbf H,e_3)
+
\sinh\varphi\,
\widehat g(\mathbf H,e_4).
\end{equation}
Thus
\begin{equation}
\label{eq:Hamiltonian-boost-expanded}
\begin{split}
\mathcal H_{\rm phys}(\varphi)
={}&
-a\left(
\cosh\varphi\,
\widehat g(\mathbf H,e_3)
+
\sinh\varphi\,
\widehat g(\mathbf H,e_4)
\right)
\\
&-
\widehat g
\left(
\widehat\nabla_{\nabla\tau}e_3,e_4
\right)
+
\nabla\tau(\varphi).
\end{split}
\end{equation}
After integration over the closed surface $\Sigma$, the last term may
be integrated by parts:
\begin{equation}
\label{eq:boost-integration-by-parts}
\int_\Sigma\nabla\tau(\varphi)\,d\mu_\sigma
=
-\int_\Sigma
\varphi\,\Delta_\sigma\tau\,d\mu_\sigma.
\end{equation}
Hence the dependence of the physical Hamiltonian on the normal-frame
choice is reduced to a scalar variational problem for the boost
parameter $\varphi$.

\noindent The crucial observation of Wang and Yau is that, when the mean
curvature vector $\mathbf H$ is spacelike, this boost-dependent
Hamiltonian possesses a canonical minimizing gauge.  Equivalently, the
normal-frame redundancy can be eliminated by minimizing the physical
Hamiltonian contribution over the proper orthochronous boosts of
$T^\perp\Sigma$.  The Euler--Lagrange equation of this pointwise
convex minimization determines the boost parameter uniquely and selects
the canonical normal frame used in the definition of the Wang--Yau
energy.  In particular, the resulting quasi-local energy is independent
of the auxiliary choice of orthonormal frame in the normal bundle.

\noindent An entirely analogous gauge fixing is performed for the reference
normal frame along $i_0(\Sigma)\subset\mathbb R^{1,3}$.  We shall
therefore always understand the normal frames entering the Wang--Yau
functional to be chosen in their corresponding canonical gauges.  We
refer to Section~2 of \cite{yau} for the detailed derivation of the
boost minimization and the associated canonical gauge condition.

\noindent This Wang-Yau quasi-local mass possesses several good properties that are desired on physical ground. This mass is strictly positive for $2-$surfaces bounding a space-like domain in a curved spacetime that satisfies the dominant energy condition and it identically vanishes for any such $2-$surface in Minkowski spacetime. In addition, it coincides with the ADM mass at space-like infinity \cite{wang2010limit} and Bondi-mass at null infinity \cite{chen2011evaluating} and reproduces the time component of the Bel-Robinson tensor (a pure gravitational entity) together with the matter stress-energy at the small sphere limit, that is when the $2-$sphere of interest is evolved by the flow of its null geodesic generators and the vertex of the associated null cone is approached \cite{chen2018evaluating}. In addition, explicit conservation laws were also discovered at the asymptotic infinity \cite{chen2015conserved}.

\noindent The Wang--Yau energy admits a reformulation in terms of purely
Riemannian boundary geometry associated with the Jang deformation.
This reduction is a central ingredient in the positivity argument and
proceeds in two steps, corresponding respectively to the reference and
physical contributions to the Wang--Yau Hamiltonian.  We summarize the
construction here and defer the detailed derivations to
propositions~\ref{prop:reference-Hamiltonian-reduction} and \ref{mean2}.

\noindent We first consider the reference term. Let
\[
i_{0}:(\Sigma,\sigma)\longrightarrow
(\mathbb{R}^{1,3},\eta)
\]
be an isometric embedding and write
\[
i_{0}(\Sigma)
=
\bigl(X^{0},X^{1},X^{2},X^{3}\bigr).
\]
Fix a future-directed constant unit timelike vector
$T_{0}\in\mathbb{R}^{1,3}$ and define the corresponding time function
on $\Sigma$ by
\[
\tau:=-\langle X,T_{0}\rangle_{\eta}.
\]
Let
\[
\widehat{i}_{0}(\widehat{\Sigma})
\subset
T_{0}^{\perp}\simeq\mathbb{R}^{3}
\]
denote the orthogonal projection of $i_{0}(\Sigma)$ onto the spacelike
hyperplane orthogonal to $T_{0}$.  The induced metric on the projected
surface is
\begin{equation}
\label{eq:projected-sigma}
\widehat{\sigma}
=
\sigma+d\tau\otimes d\tau.
\end{equation}
Let $k_{0}$ denote the mean curvature of
$\widehat{i}_{0}(\widehat{\Sigma})$ in
$T_{0}^{\perp}\simeq\mathbb{R}^{3}$, computed with respect to the
outward Euclidean unit normal.  The reference Hamiltonian contribution
then satisfies the identity
\begin{equation}
\label{eq:red1}
\int_{\widehat{i}_{0}(\widehat{\Sigma})}
k_{0}\,d\mu_{\widehat{\sigma}}
=
\int_{\Sigma}
\left(
-\sqrt{1+|\nabla\tau|_{\sigma}^{2}}\,
\langle\mathbf{H}_{0},e_{30}\rangle_{\eta}
-
\left\langle
\nabla^{\eta}_{\nabla\tau}e_{30},
e_{40}
\right\rangle_{\eta}
\right)
d\mu_{\sigma}.
\end{equation}
Thus the full reference contribution to the Wang--Yau energy is
identified with the total mean curvature of the Euclidean projection
of the reference embedding.  The proof of \eqref{eq:red1} is given in
Theorem~\ref{prop:reference-Hamiltonian-reduction}.

\noindent The reduction of the physical contribution requires the Jang
deformation.  Let $(\Omega,g,p)$ be a spacelike initial data set in the
physical spacetime $(M,\widehat g)$, with
\[
\partial\Omega=\Sigma,
\]
where $g$ is the induced Riemannian metric on $\Omega$ and $p$ is the
second fundamental form of $\Omega$ in $(M,\widehat g)$ with respect to
a chosen future-directed timelike unit normal.  Consider the Riemannian
product
\[
(\Omega\times\mathbb R,g+dt^{2})
\]
and extend $p$ trivially in the $\mathbb R$--direction to a symmetric
$2$--tensor $P$ on $\Omega\times\mathbb R$.  Thus, if $v$ denotes the
unit vector tangent to the $\mathbb R$--factor, then
\begin{equation}
\label{eq:P-extension}
P(v,\cdot)=0,
\qquad
P|_{T\Omega\times T\Omega}=p.
\end{equation}

\noindent Let
\[
f:\Omega\longrightarrow\mathbb R
\]
be a function satisfying
\[
f|_{\Sigma}=\tau,
\]
and consider its graph
\begin{equation}
\label{eq:Jang-graph-definition}
\widehat{\Omega}
:=
\left\{
(x,f(x)):x\in\Omega
\right\}
\subset
\Omega\times\mathbb R.
\end{equation}
The metric induced on $\widehat{\Omega}$ is
\begin{equation}
\label{eq:Jang-metric}
\widehat g_{ij}
=
g_{ij}+f_{i}f_{j}.
\end{equation}
Jang's equation prescribes that the mean curvature of
$\widehat{\Omega}$ in the product manifold agree with the trace of the
restriction of $P$ to $T\widehat{\Omega}$.  In local coordinates this
takes the form
\begin{equation}
\label{eq:Jang2}
\left(
g^{ij}
-
\frac{f^{i}f^{j}}
{1+|\widetilde{\nabla}f|_{g}^{2}}
\right)
\left(
\frac{
\widetilde{\nabla}_{i}\widetilde{\nabla}_{j}f
}{
\sqrt{1+|\widetilde{\nabla}f|_{g}^{2}}
}
-
P_{ij}
\right)
=
0
\qquad\text{in }\Omega,
\end{equation}
together with the Dirichlet condition
\begin{equation}
\label{eq:Jang-Dirichlet}
f=\tau
\qquad\text{on }\Sigma.
\end{equation}
Here $\widetilde{\nabla}$ denotes the Levi--Civita connection of $g$,
and
\[
f^{i}:=g^{ij}\widetilde{\nabla}_{j}f.
\]
The analysis of Jang's equation originates in the work of
Schoen--Yau \cite{schoen1979proof,schoen1981proof} in their proof of the
positive mass theorem.  The Dirichlet problem relevant to the
Wang--Yau construction is treated in \cite{yau}.

\noindent Assume that \eqref{eq:Jang2}--\eqref{eq:Jang-Dirichlet} admits a smooth
solution.  Then
$\widehat{\Sigma}:=\partial\widehat{\Omega}$
is the graph of $\tau$ over $\Sigma$, and its induced metric is
\begin{equation}
\label{eq:Jang-boundary-metric}
\widehat g|_{T\widehat{\Sigma}}
=
\sigma+d\tau\otimes d\tau
=
\widehat{\sigma}.
\end{equation}
Consequently,
$(\widehat{\Sigma},\widehat{\sigma})$
is isometric to the Euclidean projected surface
$\widehat{i}_{0}(\widehat{\Sigma})$.  This identity of the boundary
metrics is essential: it allows the physical and reference
Hamiltonians to be expressed on the same Riemannian $2$--surface.

\noindent Let
$\{\widetilde e_{\alpha}\}_{\alpha=1}^{4}$
be an orthonormal frame along $\widehat{\Omega}$ in
$\Omega\times\mathbb R$, chosen so that
\[
\{\widetilde e_{i}\}_{i=1}^{3}
\subset T\widehat{\Omega},
\]
while $\widetilde e_{4}$ is the downward unit normal to
$\widehat{\Omega}$.  Along $\widehat{\Sigma}$, let
$\widetilde e_{3}$ denote the outward unit normal to
$\widehat{\Sigma}$ within $\widehat{\Omega}$, and let
$\widetilde k$ be the corresponding mean curvature.

\noindent The normal $e_{3}'$ along the physical surface $\Sigma$ is chosen in
the canonical gauge determined by the Jang solution.  With this choice,
the physical Hamiltonian contribution satisfies
\begin{equation}
\label{eq:red2}
\begin{split}
&
\int_{\Sigma}
\left(
-\sqrt{1+|\nabla\tau|_{\sigma}^{2}}\,
\langle\mathbf H,e_{3}'\rangle_{\widehat g}
-
\alpha_{e_{3}'}(\nabla\tau)
\right)
d\mu_{\sigma}
\\
&\qquad =
\int_{\widehat{\Sigma}}
\left(
\widetilde k
-
\left\langle
\widetilde{\nabla}_{\widetilde e_{4}}
\widetilde e_{4},
\widetilde e_{3}
\right\rangle
+
P(\widetilde e_{4},\widetilde e_{3})
\right)
d\mu_{\widehat{\sigma}}.
\end{split}
\end{equation}
The proof of \eqref{eq:red2} is given in
Theorem~\ref{mean2}. Combining \eqref{eq:red1} and \eqref{eq:red2}, the reduced Wang-Yau
energy admits the representation
\begin{equation}
\label{eq:reduced}
\begin{split}
8\pi E_{WY}(\Sigma,\tau)
={}&
\int_{\widehat{\Sigma}}
k_{0}\,d\mu_{\widehat{\sigma}}
-
\int_{\widehat{\Sigma}}
\left(
\widetilde k
-
\left\langle
\widetilde{\nabla}_{\widetilde e_{4}}
\widetilde e_{4},
\widetilde e_{3}
\right\rangle
+
P(\widetilde e_{4},\widetilde e_{3})
\right)
d\mu_{\widehat{\sigma}}.
\end{split}
\end{equation}

\noindent It is convenient to introduce the $1$--form, equivalently the vector
field, $X$ on $\widehat{\Omega}$ by
\begin{equation}
\label{eq:Jang-X-definition}
X^{\flat}
:=
\left\langle
\widetilde{\nabla}_{\widetilde e_{4}}
\widetilde e_{4},
\,\cdot\,
\right\rangle
-
P(\widetilde e_{4},\,\cdot\,).
\end{equation}
Then
\[
\left\langle X,\widetilde e_{3}\right\rangle
=
\left\langle
\widetilde{\nabla}_{\widetilde e_{4}}
\widetilde e_{4},
\widetilde e_{3}
\right\rangle
-
P(\widetilde e_{4},\widetilde e_{3}),
\]
and hence \eqref{eq:reduced} becomes
\begin{equation}
\label{eq:reduced-X}
8\pi E_{WY}(\Sigma,\tau)
=
\int_{\widehat{\Sigma}}
\left(
k_{0}
-
\widetilde k
+
\langle X,\widetilde e_{3}\rangle
\right)
d\mu_{\widehat{\sigma}}.
\end{equation}
Defining the generalized mean curvature of the Jang boundary by
\begin{equation}
\label{eq:Jang-generalized-mean}
\mathcal H
:=
\widetilde k-\langle X,\widetilde e_{3}\rangle,
\end{equation}
we obtain the particularly useful form
\begin{equation}
\label{eq:WY-reduced-final-form}
8\pi E_{WY}(\Sigma,\tau)
=
\int_{\widehat{\Sigma}}
(k_{0}-\mathcal H)\,
d\mu_{\widehat{\sigma}}.
\end{equation}

\noindent 
The reduction \eqref{eq:WY-reduced-final-form} converts the
Lorentzian Hamiltonian expression defining the Wang--Yau energy into a
comparison between the Euclidean reference mean curvature $k_{0}$ and
the generalized boundary mean curvature associated with the Jang
fill-in.  This is the form that will be used in the subsequent
positivity argument.  More precisely, our objective is to prove
\[
E_{WY}(\Sigma,\tau)\geq0
\]
for every admissible time function $\tau$ belonging to the class under
consideration.  Whenever this estimate holds uniformly over the entire
admissible class $\mathcal A(\Sigma)$, it follows immediately that
\[
M_{\rm WY}(\Sigma)
=
\inf_{\tau\in\mathcal A(\Sigma)}
E_{WY}(\Sigma,\tau)
\geq0.
\]
\begin{remark}
If the Gauss curvature of $(\Sigma, \sigma)$ is positive then an embedding into $\mathbb{R}^{3}$ with $\tau=0$ is admissible. In such case, $M^{WY}$ reduces to $M^{LY}$ and $(\Omega,\Sigma)$ coincide with $(\widehat{\Omega},\widehat{\Sigma})$. 
\end{remark} 

\subsection{Main Open Problem}
\noindent A central problem in the theory of quasi-local mass is to establish
positivity and characterize the corresponding rigidity regime by
arguments that depend only on the geometry of the bounded region and
its boundary.  For the Brown--York mass, Shi and Tam
\cite{shi2002positive} proved non-negativity under the assumptions that
$(\Omega,g)$ has non-negative scalar curvature and that
\[
\Sigma:=\partial\Omega
\]
has strictly positive Gaussian curvature and positive mean curvature.
Their argument proceeds through the construction of a suitable
asymptotically flat extension of $(\Omega,g)$.  More precisely, using
the Euclidean isometric embedding
\[
i_{0}:(\Sigma,\sigma)\hookrightarrow\mathbb R^{3},
\]
they attach to $\Sigma$ an exterior region endowed with a
quasi-spherical metric of the type introduced in Bartnik's extension
framework \cite{bartnik1993quasi}.  The metric in the exterior is
determined by a nonlinear parabolic equation chosen so that the scalar
curvature of the extension is non-negative and the induced metric and
mean curvature agree across the gluing hypersurface.  This produces a
complete asymptotically flat manifold without a corner at $\Sigma$,
or, equivalently, with matching first-order boundary geometry across
the interface.

\noindent The positivity of the Brown--York mass is then obtained by combining
the monotonicity formula associated with the quasi-spherical foliation
of the exterior region with the asymptotic identification of the
corresponding boundary functional with the ADM mass.  The positive mass
theorem of Schoen--Yau
\cite{schoen1979proof,schoen1981proof} implies that the limiting ADM
mass is non-negative, and the monotonicity along the exterior
foliation propagates this sign back to the original boundary
$\Sigma$.  Variants of this extension strategy were subsequently used
by Liu and Yau
\cite{liu2003positivity,liu2006positivity} and by Wang and Yau
\cite{yau} in proving positivity results for their respective
quasi-local mass constructions, after first reducing the Lorentzian
boundary data to an appropriate Riemannian problem.

\noindent This raises a conceptually different question: whether positivity can
be proved by an argument that is genuinely quasi-local, in the sense
that it uses only the geometric data on the compact fill-in
$\Omega$ and its boundary $\Sigma$, without adjoining an artificial
asymptotically flat end and without invoking the global positive mass
theorem.  In this direction, Schoen \cite{schoen1} formulated the
problem of finding a quasi-local proof of non-negativity for the
Brown--York mass and related quasi-local mass functionals.  The purpose
of the present work is to address this problem perturbatively: for a
class of data sufficiently close to the Minkowski configuration, we
derive positivity directly from the geometry of the compact Jang
fill-in and its boundary, with no appeal to an asymptotically flat
extension or to the positive mass theorem.

\noindent
A completely quasi-local proof of the non-negativity of the
Wang--Yau, Brown--York, or Liu--Yau mass in full generality appears to
be substantially more delicate than the corresponding argument obtained
by asymptotically flat extension, and the problem remains open in this
generality.  Several spinorial approaches have been proposed in this
direction.  In particular, \cite{montiel2022compact} proposed a spinorial
argument for the Brown--York mass.  The difficulty in such an argument
is not the interior Bochner identity itself, but rather the boundary
value problem required to convert the resulting spinorial inequality
into the \emph{unweighted} Brown--York functional.  A related
construction was developed in \cite{pyau}, where a spinor-weighted
quasi-local functional was introduced.  The latter enjoys a purely
quasi-local positivity property, but agrees with the Wang--Yau energy
only under an additional boundary normalization condition on the
harmonic spinor.  As we explain below, this normalization is not
provided by the elliptic boundary problem and constitutes the essential
obstruction to obtaining the full quasi-local mass by this method.

\noindent We describe the issue at the level of the underlying elliptic problem.
Let $(\Omega,g)$ be a compact oriented Riemannian spin $3$--manifold
with smooth boundary
\[
\Sigma=\partial\Omega,
\]
and let $\mathbb S\Omega$ denote its spinor bundle.  We write
$\nabla$ for the spin connection,
\[
\mathscr D
=
\sum_{j=1}^{3}c(e_j)\nabla_{e_j}
\]
for the Dirac operator on $\Omega$, and $D_{\Sigma}$ for the intrinsic
Dirac operator induced on $\Sigma$.  With a fixed convention for the
outward unit normal $\nu$, the Schr\"odinger--Lichnerowicz formula is
\begin{equation}
\label{eq:SL-spin-obstruction}
\mathscr D^{\,2}
=
\nabla^{*}\nabla+\frac14R_g.
\end{equation}
Combining \eqref{eq:SL-spin-obstruction} with the spinorial Gauss
formula and integrating by parts gives the spinorial Reilly identity
\begin{equation}
\label{eq:spin-Reilly-obstruction}
\begin{split}
\int_{\Sigma}
\left(
\left\langle
D_{\Sigma}\psi,\psi
\right\rangle
-\frac{H}{2}|\psi|^{2}
\right)d\mu_{\sigma}
=
\int_{\Omega}
\left(
|\nabla\psi|^{2}
+\frac14R_g|\psi|^{2}
-
|\mathscr D\psi|^{2}
\right)d\mu_g .
\end{split}
\end{equation}
Here $H$ denotes the mean curvature of $\Sigma$ with respect to the
normal convention used in \eqref{eq:spin-Reilly-obstruction}.

The term
\begin{equation}
\label{eq:bad-Dirac-term}
-\int_{\Omega}|\mathscr D\psi|^{2}\,d\mu_g
\end{equation}
has the wrong sign for a direct positivity argument.  The natural
strategy is therefore to require
\begin{equation}
\label{eq:harmonic-spinor}
\mathscr D\psi=0
\qquad\text{in }\Omega.
\end{equation}
For example, when $R_g\geq0$, equations
\eqref{eq:spin-Reilly-obstruction}--\eqref{eq:harmonic-spinor} imply
\begin{equation}
\label{eq:spin-Reilly-positive}
\int_{\Sigma}
\left(
\left\langle
D_{\Sigma}\psi,\psi
\right\rangle
-\frac{H}{2}|\psi|^{2}
\right)d\mu_{\sigma}
\geq0.
\end{equation}
The analogous inequality on the Jang fill-in contains the generalized
mean curvature occurring in the Wang--Yau reduction.  Indeed, if a
vector field $X$ satisfies
\begin{equation}
\label{eq:Jang-spin-curvature}
R_g
\geq
2|X|^{2}
-
2\operatorname{div}_gX,
\end{equation}
then \eqref{eq:spin-Reilly-obstruction} and
\eqref{eq:harmonic-spinor} yield
\begin{equation}
\label{eq:spin-Reilly-X}
\int_{\Sigma}
\left\langle
D_{\Sigma}\psi,\psi
\right\rangle
d\mu_{\sigma}
\geq
\frac12
\int_{\Sigma}
\left(
H-\langle X,\nu\rangle
\right)|\psi|^{2}
d\mu_{\sigma}.
\end{equation}
To see the sign in \eqref{eq:spin-Reilly-X}, one uses
\eqref{eq:Jang-spin-curvature} to estimate the right hand side of
\eqref{eq:spin-Reilly-obstruction} by
\[
\begin{split}
&\int_{\Omega}
\left(
|\nabla\psi|^{2}
+
\frac12|X|^{2}|\psi|^{2}
-
\frac12(\operatorname{div}X)|\psi|^{2}
\right)d\mu_g .
\end{split}
\]
Integration by parts gives
\begin{equation}
\label{eq:spin-X-integration}
\begin{split}
-\frac12
\int_{\Omega}
(\operatorname{div}X)|\psi|^{2}\,d\mu_g
={}&
-\frac12
\int_{\Sigma}
\langle X,\nu\rangle|\psi|^{2}\,d\mu_{\sigma}
\\
&+
\frac12
\int_{\Omega}
X(|\psi|^{2})\,d\mu_g,
\end{split}
\end{equation}
while metric compatibility of the spin connection gives
\[
\frac12X(|\psi|^{2})
=
\operatorname{Re}
\langle\nabla_X\psi,\psi\rangle.
\]
Consequently, the remaining interior integrand is
\[
|\nabla\psi|^{2}
+
\operatorname{Re}
\langle\nabla_X\psi,\psi\rangle
+
\frac12|X|^{2}|\psi|^{2}.
\]
Pointwise,
\begin{equation}
\label{eq:spin-X-square}
\begin{split}
&
|\nabla\psi|^{2}
+
\operatorname{Re}
\langle\nabla_X\psi,\psi\rangle
+
\frac12|X|^{2}|\psi|^{2}
\\
&\qquad =
\left|
\nabla\psi+\frac12X^{\flat}\otimes\psi
\right|^{2}
+
\frac14|X|^{2}|\psi|^{2}
\geq0,
\end{split}
\end{equation}
which proves \eqref{eq:spin-Reilly-X}.

\noindent At this stage the interior geometric inequality has been completely
exploited; the remaining difficulty is entirely a boundary issue.
Suppose, for instance, that $(\Sigma,\sigma)$ admits an isometric
embedding into $\mathbb R^{3}$ with Euclidean mean curvature $k_0$.
A nonzero parallel spinor $\Phi$ on the Euclidean reference domain
restricts to a spinor
\[
\xi:=\Phi|_{\Sigma}
\]
satisfying the extrinsic Dirac identity
\begin{equation}
\label{eq:reference-spinor-Dirac}
D_{\Sigma}\xi
=
\frac{k_0}{2}\xi,
\qquad
|\xi|\equiv1,
\end{equation}
after normalization of $\Phi$.  If one were able to solve
\eqref{eq:harmonic-spinor} subject to the full boundary condition
\begin{equation}
\label{eq:full-spin-Dirichlet}
\psi|_{\Sigma}=\xi,
\end{equation}
then \eqref{eq:spin-Reilly-positive} and
\eqref{eq:reference-spinor-Dirac} would immediately give
\[
0
\leq
\frac12
\int_{\Sigma}
(k_0-H)|\xi|^{2}\,d\mu_{\sigma}
=
4\pi m_{\rm BY}(\Sigma),
\]
which would constitute the desired compact proof.

\noindent The difficulty is that \eqref{eq:full-spin-Dirichlet} is not an
elliptic boundary condition for the first-order operator $\mathscr D$.
For a Dirac-type operator on a compact manifold with boundary, an
elliptic boundary condition prescribes only one half of the boundary
degrees of freedom.  Equivalently, the trace of a harmonic spinor is
constrained to belong to the Cauchy data space
\[
\mathcal C_{\mathscr D}
\subset
H^{1/2}(\Sigma;\mathbb S\Sigma),
\]
which is the range of the Calder\'on projector associated with
$\mathscr D$.  Thus a prescribed boundary spinor $\xi$ can occur as
the complete trace of a solution of
\[
\mathscr D\psi=0
\]
only if
\[
\xi\in\mathcal C_{\mathscr D},
\]
a compatibility condition which is not satisfied by generic boundary
data.

\noindent For example, the local MIT bag condition is expressed through the
orthogonal projections
\begin{equation}
\label{eq:MIT-projectors-obstruction}
\Pi_{\pm}
:=
\frac12
\left(
{\rm Id}\pm\sqrt{-1}\,c(\nu)
\right),
\end{equation}
up to the choice of normal and Clifford conventions.  The elliptic
boundary problem prescribes only
\begin{equation}
\label{eq:MIT-half-boundary}
\Pi_{+}\psi
=
\Pi_{+}\xi
\qquad\text{on }\Sigma,
\end{equation}
while the complementary component $\Pi_{-}\psi$ is determined by the
interior equation.  In particular,
\begin{equation}
\label{eq:boundary-norm-decomposition}
|\psi|^{2}
=
|\Pi_{+}\xi|^{2}
+
|\Pi_{-}\psi|^{2}
\qquad\text{on }\Sigma,
\end{equation}
and there is no mechanism in
\eqref{eq:harmonic-spinor}--\eqref{eq:MIT-half-boundary} forcing
\begin{equation}
\label{eq:constant-boundary-norm}
|\psi|\equiv1
\qquad\text{on }\Sigma.
\end{equation}
Imposing \eqref{eq:constant-boundary-norm} in addition to the elliptic
half-boundary condition amounts to introducing a further nonlinear
pointwise constraint on the Cauchy data.  Such a constraint is not a
consequence of the Dirac boundary-value problem and, on a generic
geometry, there is no reason for it to be solvable.

\noindent This distinction is decisive for the quasi-local mass.  In the
Wang--Yau setting, after the Jang reduction one has
\begin{equation}
\label{eq:WY-unweighted-spin-comparison}
8\pi E_{WY}(\Sigma,\tau)
=
\int_{\widehat\Sigma}
\left(
k_0-\mathcal H
\right)d\mu_{\widehat\sigma},
\qquad
\mathcal H
=
\widetilde k-\langle X,\nu\rangle.
\end{equation}
A spinorial argument of the preceding type naturally controls instead
the weighted quantity
\begin{equation}
\label{eq:weighted-spin-WY}
8\pi M_{\tau}^{\psi}
:=
\int_{\widehat\Sigma}
\left(
k_0-\mathcal H
\right)
|\psi|^{2}
d\mu_{\widehat\sigma}.
\end{equation}
Under the appropriate curvature inequality and elliptic boundary
condition one may obtain
\begin{equation}
\label{eq:weighted-spin-positive}
M_{\tau}^{\psi}\geq0.
\end{equation}
However, \eqref{eq:weighted-spin-positive} does \emph{not} imply
\[
E_{WY}(\Sigma,\tau)\geq0,
\]
because the boundary density
\[
k_0-\mathcal H
\]
has no a priori pointwise sign.  Positivity of one weighted integral
therefore contains strictly less information than positivity of the
corresponding unweighted integral.  The two expressions agree
directly only under the normalization
\begin{equation}
\label{eq:spin-normalization-needed}
|\psi|^{2}\equiv1
\qquad\text{on }\widehat\Sigma,
\end{equation}
which is precisely the additional condition that is not furnished by
the elliptic Dirac problem.  This is the essential limitation of the
spinorial construction in \cite{pyau} when one attempts to recover the
Wang--Yau mass itself rather than a spinor-weighted quasi-local
functional.

\noindent It is useful to distinguish this boundary normalization problem from
the rigidity of the equality case.  A parallel spinor does satisfy
\[
\nabla\psi=0
\qquad\Longrightarrow\qquad
|\psi|\equiv {\rm constant},
\]
but the existence of a nonzero parallel spinor is already a highly
rigid geometric condition.  Indeed it leads to the Ricci flatness
\begin{equation}
\label{eq:parallel-implies-Ricci-flat}
\operatorname{Ric}_g=0.
\end{equation}
In dimension three the Riemann curvature tensor is determined
algebraically by the Ricci tensor and
therefore \eqref{eq:parallel-implies-Ricci-flat} implies
\[
\operatorname{Rm}_g=0.
\]
Thus the parallel-spinor mechanism naturally occurs in the rigidity
regime itself and cannot be used as a generic device for enforcing the
boundary normalization \eqref{eq:spin-normalization-needed}.

\noindent Accordingly, the obstruction in the spinorial approach is not the
absence of a Bochner inequality or of an elliptic boundary condition
for the Dirac operator.  Both are available.  The obstruction is the
mismatch between the admissible Cauchy data of a first-order elliptic
system and the boundary normalization required to remove the spinorial
weight from the quasi-local Hamiltonian.  This is why the spinorial
argument naturally produces the positive quantity
\eqref{eq:weighted-spin-WY}, while recovering the Wang--Yau,
Brown--York, or Liu--Yau mass itself requires an additional condition
which is not available on a generic compact fill-in. First, we define the following space of the transverse-traceless tensors

\subsection{Main Result}
\noindent In this section we provide the main theorem that we prove in this article. The main result of this article is split into three theorems. First, we define the following space 
\[
{\mathcal T}^{k,\alpha}
:=
\left\{
h\in C^{k,\alpha}
(\Omega_0;S^2T^*\Omega_0): ||h||_{C^{k,\alpha}}=1,
\operatorname{div}_{\delta}h=0,\;
\operatorname{tr}_{\delta}h=0,\;
h^T|_{\Sigma_0}=0
\right\}.
\]

\begin{theorem}
\label{1}
Let $(M,g,k)$ be a smooth physical initial data set and suppose that the
dominant energy condition holds. Let $\tau$ be an admissible minimizing
Wang--Yau time function, and let
$(\widehat{\Omega},\widehat{g})$ be the corresponding Jang fill-in.
Let $U$ be the unique positive solution of
\begin{equation}
\label{eq:U-main}
-8\Delta_{\widehat{g}}U+R_{\widehat{g}}U=0
\quad\text{in }\widehat{\Omega},
\qquad
U=1
\quad\text{on }\partial\widehat{\Omega},
\end{equation}
and define
$g_{\mathrm{sf}}:=U^{4}\widehat{g}$.
Then
$R_{g_{\mathrm{sf}}}=0$
and
\begin{equation}
\label{eq:WY-BY-main}
E_{\mathrm{WY}}(\Sigma,\tau)
\geq
m_{\mathrm{BY}}
\left(
\partial\widehat{\Omega};
\widehat{\Omega},g_{\mathrm{sf}}
\right).
\end{equation}
\end{theorem}
\begin{theorem}
\label{2}
Let $(\Omega_{0},\delta)\subset\mathbb{R}^{3}$ be the strictly convex
Euclidean domain determined by the projected reference embedding, and
fix $k\geq 2$ and $\alpha\in(0,1)$. Then there exists a nontrivial
class of physical initial data $(M,g,k)$ satisfying the
Einstein constraint equations and the dominant energy condition for
which the corresponding scalar-flat Jang reduction $g_{sf}$ admits the following
description.

\noindent There exist
a boundary-preserving diffeomorphism
$\Phi:\Omega_0\longrightarrow\widehat\Omega$,
a nonzero TT tensor $h$, and a sufficiently small parameter $\lambda$
such that
\[
\Phi^*g_{\rm sf}
=
G(\lambda h)
=
v_\lambda^4(\delta+\lambda h).
\]
where $v_{\lambda}$ is the unique positive solution of
\begin{equation}
\label{eq:v-main}
-8\Delta_{\delta+\lambda h}v_{\lambda}
+
R_{\delta+\lambda h}v_{\lambda}
=
0
\quad\text{in }\Omega_{0},
\qquad
v_{\lambda}=1
\quad\text{on }\partial\Omega_{0}.
\end{equation}
Moreover,
$\|v_{\lambda}-1\|_{C^{k,\alpha}(\Omega_{0})}
\leq
C_{h}\lambda^{2}$,
and consequently
\begin{equation}
\label{eq:gsf-expansion-main}
\Phi^{*}g_{\mathrm{sf}}
=
\delta+\lambda h
+
O_{C^{k,\alpha}}(\lambda^{2}).
\end{equation}
\end{theorem}
\begin{theorem}
Let $(M,g,k)$ satisfy the dominant energy condition, and let $\tau_*$ be
an admissible minimizing Wang--Yau time function such that the pair
$\bigl((M,g,k),\tau_*\bigr)$
belongs to the TT-generated Jang-reduced class of Definition~6. Then,
for every fixed nonzero
$h\in\mathcal{T}^{k,\alpha}$,
there exists $\varepsilon(h)>0$ such that, whenever
$0<|\lambda|<\varepsilon(h)$,
one has
$m_{\rm BY}
\bigl(
\partial\widehat{\Omega}_*;
\widehat{\Omega}_*,
g_{{\rm sf},*}
\bigr)>0$.
Consequently,
$E_{\rm WY}(\Sigma,\tau_*)>0,
\qquad
M_{\rm WY}(\Sigma)
=
E_{\rm WY}(\Sigma,\tau_*)>0$.
\end{theorem}

\subsection{Outline of the proof}
\label{subsec:outline-proof}
\noindent The proof of Theorems~(\ref{1})-(\ref{3}) proceeds through four
geometric reductions.  The essential point is that the Lorentzian
Wang--Yau problem is first converted into a Riemannian boundary problem
on the Jang graph; the latter is then reduced, by a conformal
deformation, to the Brown--York functional of a scalar-flat fill-in.
In the small-data regime this scalar-flat metric is represented as the
scalar-flat conformal completion of a transverse--traceless perturbation
of a Euclidean domain, and the desired positivity follows from the
strict positivity of the Brown--York Hessian in TT directions. Below is the summary of the main proof idea presented in step by step while the complete proof is presented through a sequence of propositions in sections \ref{3}-\ref{5}.

\medskip
\noindent
Let $(\Omega,g,K)$ denote the physical initial data bounded by
$\Sigma=\partial\Omega$, and fix an admissible Wang--Yau time function
$\tau$.  We solve Jang's equation on $\Omega$ with Dirichlet data
\[
        f|_{\Sigma}=\tau
\]
and denote by
\[
        (\widehat\Omega,\widehat g)
\]
the resulting Jang graph.  Its boundary
$\widehat\Sigma:=\partial\widehat\Omega$
has induced metric
       $\widehat\sigma
        =
        \sigma+d\tau\otimes d\tau.$
On the reference side, if the Minkowski isometric embedding of
$(\Sigma,\sigma)$ is projected orthogonally onto the spacelike
hyperplane orthogonal to the reference observer $T_0$, the resulting
Euclidean surface has precisely the same induced metric
$\widehat\sigma$.  Denoting its Euclidean mean curvature by $k_0$, the
reference Hamiltonian reduces to
\[
        \int_{\widehat\Sigma}
        k_0\,dA_{\widehat\sigma}.
\]

\noindent For the physical term, let $\widetilde e_4$ denote the downward unit
normal to the Jang graph and $\widetilde e_3$ the outward unit normal to
$\widehat\Sigma$ in $\widehat\Omega$.  Define the Jang vector field
$X$ by
\begin{equation}
\label{eq:outline-X}
        X^\flat
        :=
        \left\langle
        \widetilde\nabla_{\widetilde e_4}\widetilde e_4,
        \,\cdot\,
        \right\rangle
        -
        P(\widetilde e_4,\,\cdot\,),
\end{equation}
where $P$ denotes the extension of $K$ to the product used in the Jang
construction.  Thus
\[
        \langle X,\widetilde e_3\rangle
        =
        \left\langle
        \widetilde\nabla_{\widetilde e_4}\widetilde e_4,
        \widetilde e_3
        \right\rangle
        -
        P(\widetilde e_4,\widetilde e_3).
\]
If $\widetilde k$ denotes the mean curvature of
$\widehat\Sigma\subset\widehat\Omega$, the Wang--Yau reduction therefore
takes the form
\begin{equation}
\label{eq:outline-WY-Jang}
\begin{split}
        8\pi E_{\rm WY}(\Sigma,\tau)
        &=
        \int_{\widehat\Sigma}k_0\,dA_{\widehat\sigma}
        -
        \int_{\widehat\Sigma}
        \left(
        \widetilde k-\langle X,\widetilde e_3\rangle
        \right)
        dA_{\widehat\sigma}
\\
        &=
        \int_{\widehat\Sigma}
        \left(k_0-\mathcal H\right)
        dA_{\widehat\sigma},
\end{split}
\end{equation}
where
\begin{equation}
\label{eq:outline-generalized-H}
        \mathcal H
        :=
        \widetilde k-\langle X,\widetilde e_3\rangle
\end{equation}
is the generalized mean curvature of the Jang boundary.

\noindent The momentum information of the original spacetime has thus not
disappeared; rather, it is encoded by the correction
$\langle X,\widetilde e_3\rangle$ to the ordinary Riemannian mean
curvature.  At this stage the positivity problem has become a
Riemannian problem on the compact manifold
$(\widehat\Omega,\widehat g)$.

\noindent The second fundamental identity supplied by the Jang deformation is
the scalar-curvature estimate
\begin{equation}
\label{eq:outline-Jang-curvature}
        R_{\widehat g}
        \geq
        2|X|_{\widehat g}^{2}
        -
        2\operatorname{div}_{\widehat g}X,
\end{equation}
which follows from the constraint equations and the dominant energy
condition.  This inequality is the input for the next step.

\medskip
\noindent
We solve the Dirichlet problem
\begin{equation}
\label{eq:outline-U-problem}
\left\{
\begin{aligned}
        -8\Delta_{\widehat g}U
        +R_{\widehat g}U&=0
        &&\text{in }\widehat\Omega,\\
        U&=1
        &&\text{on }\widehat\Sigma.
\end{aligned}
\right.
\end{equation}
The curvature inequality \eqref{eq:outline-Jang-curvature} gives the
coercive estimate
\begin{equation}
\label{eq:outline-coercivity}
\begin{split}
\int_{\widehat\Omega}
\left(
8|\widehat\nabla\phi|^2
+
R_{\widehat g}\phi^2
\right)dV_{\widehat g}
\geq{}&
6\int_{\widehat\Omega}
|\widehat\nabla\phi|^2\,dV_{\widehat g}
\\
&+
2\int_{\widehat\Omega}
|\widehat\nabla\phi+\phi X|^2\,dV_{\widehat g}
\end{split}
\end{equation}
for every $\phi\in H^1_0(\widehat\Omega)$.  Consequently the
Dirichlet conformal Laplacian is invertible, and
\eqref{eq:outline-U-problem} possesses a unique positive solution.

\noindent Define
\begin{equation}
\label{eq:outline-gsf}
        g_{\rm sf}:=U^4\widehat g.
\end{equation}
The conformal scalar-curvature identity in dimension three,
\[
        R_{U^4\widehat g}
        =
        U^{-5}
        \left(
        -8\Delta_{\widehat g}U
        +
        R_{\widehat g}U
        \right),
\]
implies
\begin{equation}
\label{eq:outline-gsf-flat}
        R_{g_{\rm sf}}=0.
\end{equation}
Moreover, since $U=1$ on $\widehat\Sigma$,
\begin{equation}
\label{eq:outline-boundary-preserved}
        g_{\rm sf}|_{T\widehat\Sigma}
        =
        \widehat\sigma.
\end{equation}
Thus the reference Euclidean embedding, and hence $k_0$, remain
unchanged.

\noindent Let $H_{\rm sf}$ denote the mean curvature of
$\widehat\Sigma\subset(\widehat\Omega,g_{\rm sf})$.  With the outward
normal convention used throughout the paper,
\begin{equation}
\label{eq:outline-H-conformal}
        H_{\rm sf}
        =
        \widetilde k+4\partial_{\nu}U
        \qquad\text{on }\widehat\Sigma.
\end{equation}
Multiplying \eqref{eq:outline-U-problem} by $U$, integrating by parts,
and using \eqref{eq:outline-Jang-curvature}, one obtains
\begin{equation}
\label{eq:outline-flux}
\begin{split}
\int_{\widehat\Sigma}
\left(
H_{\rm sf}-\mathcal H
\right)dA_{\widehat\sigma}
\geq{}&
3\int_{\widehat\Omega}
|\widehat\nabla U|^2\,dV_{\widehat g}
\\
&+
\int_{\widehat\Omega}
|\widehat\nabla U+UX|^2\,dV_{\widehat g}
\geq0.
\end{split}
\end{equation}
Consequently,
\begin{equation}
\label{eq:outline-WY-BY-decomp}
\begin{split}
8\pi E_{\rm WY}(\Sigma,\tau)
={}&
\int_{\widehat\Sigma}
(k_0-H_{\rm sf})\,dA_{\widehat\sigma}
\\
&+
\int_{\widehat\Sigma}
(H_{\rm sf}-\mathcal H)\,dA_{\widehat\sigma},
\end{split}
\end{equation}
and hence
\begin{equation}
\label{eq:outline-WY-BY}
\begin{split}
E_{\rm WY}(\Sigma,\tau)
\geq{}&
m_{\rm BY}
(\widehat\Sigma;\widehat\Omega,g_{\rm sf})
\\
&+
\frac{3}{8\pi}
\int_{\widehat\Omega}
|\widehat\nabla U|^2\,dV_{\widehat g}
+
\frac{1}{8\pi}
\int_{\widehat\Omega}
|\widehat\nabla U+UX|^2\,dV_{\widehat g}.
\end{split}
\end{equation}
In particular,
\begin{equation}
\label{eq:outline-WY-ge-BY}
        E_{\rm WY}(\Sigma,\tau)
        \geq
        m_{\rm BY}
        (\widehat\Sigma;\widehat\Omega,g_{\rm sf}).
\end{equation}
Thus the problem has been reduced to a purely Riemannian positivity
statement for the Brown--York mass of the scalar-flat metric
$g_{\rm sf}$.

\medskip
\noindent
We next analyze scalar-flat metrics sufficiently close to the Euclidean
reference fill-in.  After a boundary-preserving identification with a
strictly convex Euclidean domain
\[
        \Omega_0\Subset\mathbb R^3,
        \qquad
        \Sigma_0=\partial\Omega_0,
\]
the small-data decomposition established below expresses the
scalar-flat metric arising from Step~2 in the form
\begin{equation}
\label{eq:outline-second-conformal}
        \Phi^*g_{\rm sf}
        =
        v_\lambda^4(\delta+\lambda h),
\end{equation}
where $0<|\lambda|\ll1$ and
\begin{equation}
\label{eq:outline-TT}
        \operatorname{div}_{\delta}h=0,
        \qquad
        \operatorname{tr}_{\delta}h=0,
        \qquad
        h^T|_{\Sigma_0}=0.
\end{equation}
The second conformal factor is obtained from
\begin{equation}
\label{eq:outline-v-problem}
\left\{
\begin{aligned}
        -8\Delta_{\delta+\lambda h}v_\lambda
        +
        R_{\delta+\lambda h}v_\lambda&=0
        &&\text{in }\Omega_0,\\
        v_\lambda&=1
        &&\text{on }\Sigma_0.
\end{aligned}
\right.
\end{equation}
The TT conditions have the decisive consequence
\begin{equation}
\label{eq:outline-linear-scalar}
        DR_{\delta}(h)
        =
        -\Delta_{\delta}
        (\operatorname{tr}_{\delta}h)
        +
        \partial^i\partial^jh_{ij}
        =
        0.
\end{equation}
Differentiating \eqref{eq:outline-v-problem} at $\lambda=0$ therefore
gives
\[
        \Delta_{\delta}\dot v_0=0,
        \qquad
        \dot v_0|_{\Sigma_0}=0,
\]
and hence
\begin{equation}
\label{eq:outline-v-quadratic}
        \dot v_0=0,
        \qquad
        v_\lambda
        =
        1+O_{C^{k,\alpha}}(\lambda^2).
\end{equation}
Thus
\begin{equation}
\label{eq:outline-sf-TT-expansion}
        \Phi^*g_{\rm sf}
        =
        \delta+\lambda h+O_{C^{k,\alpha}}(\lambda^2),
\end{equation}
but, importantly, the metric on the left-hand side is exactly
scalar-flat; \eqref{eq:outline-sf-TT-expansion} is a consequence of the
exact conformal construction \eqref{eq:outline-second-conformal}, not
an independent perturbative ansatz.

\noindent Notice that, in addition to the quasilinear Dirichlet problem defining
the Jang graph, two conformal Dirichlet problems occur in the argument:
\eqref{eq:outline-U-problem}, which converts the Jang metric into a
scalar-flat metric, and \eqref{eq:outline-v-problem}, which realizes
that scalar-flat metric in the TT-conformal coordinates adapted to the
Euclidean reference configuration.

\medskip
\noindent
It remains to establish positivity of the Brown--York mass for the
scalar-flat TT-conformal family
\begin{equation}
\label{eq:outline-gt}
        g_\lambda
        :=
        v_\lambda^4(\delta+\lambda h).
\end{equation}
Since
\[
        R_{g_\lambda}=0,
        \qquad
        g_\lambda|_{T\Sigma_0}=\gamma_0,
        \qquad
        g_0=\delta,
        \qquad
        \dot g_0=h,
\]
the Euclidean metric is a critical point of the Brown--York functional:
\begin{equation}
\label{eq:outline-BY-first}
        m_{\rm BY}(g_0)=0,
        \qquad
        \left.
        \frac{d}{d\lambda}
        \right|_{\lambda=0}
        m_{\rm BY}(g_\lambda)=0.
\end{equation}
If $Z_h\in\Gamma(T\Sigma_0)$ is defined by
\[
        \gamma_0(Z_h,Y)
        =
        h(\nu_0,Y),
        \qquad
        Y\in T\Sigma_0,
\]
a direct computation of the constrained second variation gives
\begin{equation}
\label{eq:outline-BY-Hessian}
\begin{split}
8\pi
\left.
\frac{d^2}{d\lambda^2}
\right|_{\lambda=0}
m_{\rm BY}(g_\lambda)
={}&
\frac14
\int_{\Omega_0}
|\nabla^\delta h|_\delta^2\,dV_\delta
\\
&+
\frac12
\int_{\Sigma_0}
\left(
A_0(Z_h,Z_h)
+
H_0|Z_h|_{\gamma_0}^2
\right)dA_{\gamma_0}.
\end{split}
\end{equation}
Since $\Sigma_0$ is strictly convex,
\[
        A_0>0,
        \qquad
        H_0>0,
\]
and therefore the quadratic form on the right-hand side of
\eqref{eq:outline-BY-Hessian} is strictly positive for every
$0\neq h$ satisfying \eqref{eq:outline-TT}.  It follows that
\begin{equation}
\label{eq:outline-BY-expansion}
\begin{split}
m_{\rm BY}(g_\lambda)
={}&
\frac{\lambda^2}{64\pi}
\int_{\Omega_0}
|\nabla^\delta h|_\delta^2\,dV_\delta
\\
&+
\frac{\lambda^2}{32\pi}
\int_{\Sigma_0}
\left(
A_0(Z_h,Z_h)
+
H_0|Z_h|_{\gamma_0}^2
\right)dA_{\gamma_0}
+
o(\lambda^2),
\end{split}
\end{equation}
and consequently
\begin{equation}
\label{eq:outline-BY-positive}
        m_{\rm BY}(g_\lambda)>0,
        \qquad
        0<|\lambda|\ll1.
\end{equation}

\noindent Combining \eqref{eq:outline-WY-ge-BY} with
\eqref{eq:outline-BY-positive} yields
\[
        E_{\rm WY}(\Sigma,\tau)>0
\]
for every sufficiently small nontrivial perturbation in the class under
consideration.  Finally, for a minimizing admissible time function
$\tau$,
\[
        M_{\rm WY}(\Sigma)
        =
        E_{\rm WY}(\Sigma,\tau),
\]
and hence
\begin{equation}
\label{eq:outline-final-WY}
        M_{\rm WY}(\Sigma)>0.
\end{equation}
At the Minkowski configuration the two conformal factors are identically
one, the TT perturbation vanishes, and all of the inequalities above
are equalities, recovering
\[
        M_{\rm WY}(\Sigma)=0.
\]

\noindent The argument is therefore entirely quasi-local.  No asymptotically
flat extension of the compact fill-in is introduced, and the positive
mass theorem does not enter the proof. 

\section{Acknowledgment}
\noindent This work is supported by BIMSA of Tsinghua University and the Beijing Municipal Government. We thank Prof. Mu-Tao Wang for helpful discussions. 

\section{TT-generated scalar-flat fill-ins near a Euclidean domain}
\label{3}

\noindent In this section, we explicitly construct conformally scalar flat metrics that are generated by the transverse-traceless perturbations of the Euclidean metric. Here the choice of transverse-traceless perturbation can be thought of as a the most physically relevant ones since the pure gravitational degrees of freedom are essentially transverse-tracelss. Let $\Omega\Subset\mathbb R^3$ be a smooth bounded domain with boundary
$\Sigma:=\partial\Omega,$
and let $\delta$ denote the Euclidean metric. We write
\[
\gamma_0:=\delta|_{T\Sigma}.
\]
Fix $k\geq 2$ and $0<\alpha<1$. We state the following important definition and the lemma that explicitly constructs transverse-traceless tensor on a ball. 

\begin{definition}
\label{def:TT-space}
We define
\[
\mathcal T^{k,\alpha}
:=
\left\{
h\in C^{k,\alpha}(\Omega;S^2T^*\Omega):
\operatorname{div}_\delta h=0,\quad
\operatorname{tr}_\delta h=0,\quad
h^T|_\Sigma=0
\right\}.
\]
Here $h^T$ denotes the restriction of $h$ to
$T\Sigma\times T\Sigma$. Elements of $\mathcal T^{k,\alpha}$ will be called TT seeds that will be used to perturb the interior Euclidean metric.
\end{definition}
\noindent For later use, we also introduce
\[
\mathcal T^\infty_c
:=
\left\{
h\in C^\infty_c(\Omega;S^2T^*\Omega):
\operatorname{div}_\delta h=0,\quad
\operatorname{tr}_\delta h=0
\right\}.
\]

\begin{lemma}[Existence of compactly supported TT tensors]
\label{lem:compact-TT}

The space $\mathcal T^\infty_c$ is infinite dimensional.

\end{lemma}

\begin{proof}
Let $\mathcal C(g)$ denote the Cotton--York tensor of a
three-dimensional Riemannian metric $g$.  With a fixed choice of
orientation, we use the convention
\[
\mathcal C(g)_{ij}
=
\varepsilon_i{}^{k\ell}
\nabla_k
\left(
\text{Ric}_{g,\ell j}
-\frac14 R_g g_{\ell j}
\right).
\]
The Cotton--York tensor is symmetric, trace-free, and divergence-free:
\begin{equation}
\label{eq:CY-identities}
\mathcal C(g)_{ij}=\mathcal C(g)_{ji},
\qquad
\tr_g\mathcal C(g)=0,
\qquad
\div_g\mathcal C(g)=0.
\end{equation}
Since the Euclidean metric is flat,
\[
\mathcal C(\delta)=0.
\]
Denote by
\[
\mathcal C'_\delta
:
C^\infty(\Omega;S^2T^*\Omega)
\longrightarrow
C^\infty(\Omega;S^2T^*\Omega)
\]
the linearization of $\mathcal C$ at $\delta$.  Differentiating
\eqref{eq:CY-identities} at $\delta$, and using
$\mathcal C(\delta)=0$, gives
\begin{equation}
\label{eq:CYprime-TT}
\tr_\delta \mathcal C'_\delta(q)=0,
\qquad
\div_\delta \mathcal C'_\delta(q)=0
\end{equation}
for every smooth symmetric $2$-tensor $q$.  Notice that the terms
arising from differentiating the metric contraction and the
Levi--Civita connection vanish at $\delta$, since each of them is
multiplied by $\mathcal C(\delta)$.

\noindent Moreover, $\mathcal C'_\delta$ is a third-order local differential
operator.  Hence
\begin{equation}
\label{eq:CY-support}
\operatorname{supp}\mathcal C'_\delta(q)
\subset
\operatorname{supp}q .
\end{equation}

\noindent We next verify that $\mathcal C'_\delta$ is not identically zero.
At the Euclidean metric,
\[
\bigl(D\text{Ric}_\delta(q)\bigr)_{ij}
=
\frac12
\left(
\partial^a\partial_iq_{aj}
+
\partial^a\partial_jq_{ai}
-
\Delta q_{ij}
-
\partial_i\partial_j\tr_\delta q
\right)
\]
and
\[
DR_\delta(q)
=
-\Delta\tr_\delta q
+
\partial^i\partial^jq_{ij}.
\]
Consequently, we have
\begin{equation}
\label{eq:CY-linearization}
\bigl(\mathcal C'_\delta(q)\bigr)_{ij}
=
\varepsilon_i{}^{k\ell}\partial_k
\left[
\bigl(D\text{Ric}_\delta(q)\bigr)_{\ell j}
-\frac14DR_\delta(q)\delta_{\ell j}
\right].
\end{equation}

\noindent Now using the symbol convention $\partial_j\mapsto i\xi_j$, take
\[
\xi=e_1,
\qquad
Q=
\begin{pmatrix}
0&0&0\\
0&1&0\\
0&0&-1
\end{pmatrix}.
\]
Then
\[
\tr_\delta Q=0,
\qquad
\xi^jQ_{ij}=0,
\]
and therefore
\[
\sigma_2(DR_\delta)(\xi)Q=0,
\qquad
\sigma_2(D\text{Ric}_\delta)(\xi)Q
=
\frac12|\xi|^2Q.
\]
It follows from \eqref{eq:CY-linearization} that
\[
\bigl(
\sigma_3(\mathcal C'_\delta)(e_1)Q
\bigr)_{23}
=
\frac{i}{2}\neq0
\]
after fixing the orientation so that
$\varepsilon_{123}=1$.  Thus
\begin{equation}
\label{eq:CYprime-nonzero}
\mathcal C'_\delta\not\equiv0 .
\end{equation}
We now show that the construction can be localized in an arbitrary
ball. Let
$B\Subset\Omega$
be a Euclidean ball and choose $p\in B$.  Since
$\mathcal C'_\delta$ is a nonzero differential operator, there exists
a smooth symmetric $2$-tensor $q$ defined near $p$ such that
\[
\mathcal C'_\delta(q)(p)\neq0.
\]
Choose $\chi\in C_c^\infty(B)$ satisfying $\chi\equiv1$ in a
neighborhood of $p$, and put
\[
\widetilde q:=\chi q.
\]
By locality,
\[
\mathcal C'_\delta(\widetilde q)(p)
=
\mathcal C'_\delta(q)(p)\neq0.
\]
Hence
\[
h_B:=\mathcal C'_\delta(\widetilde q)
\]
is nonzero.  By \eqref{eq:CYprime-TT} and
\eqref{eq:CY-support},
\[
h_B\in C_c^\infty(B;S^2T^*\Omega),
\qquad
\div_\delta h_B=0,
\qquad
\tr_\delta h_B=0.
\]
Thus
\[
0\neq h_B\in\mathcal T_c^\infty.
\]
Finally, choose pairwise disjoint balls
\[
B_j\Subset\Omega,
\qquad j=1,2,\ldots,
\]
and construct
\[
0\neq h_j\in\mathcal T_c^\infty,
\qquad
\operatorname{supp}h_j\subset B_j.
\]
The family $\{h_j\}_{j=1}^\infty$ is linearly independent, since the
supports are pairwise disjoint.  Therefore
\[
\dim\mathcal T_c^\infty=\infty.
\]
\end{proof}

\section{Reduction of the Wang-Yau Quasilocal Energy}
\label{4}
\noindent In this section, we reduce the Wang-Yau quasi-local energy expression to the desired form \ref{eq:WY-mass-reduced-inf} to the desired form \ref{eq:WY-reduced-final-form}. As mentioned earlier, this essentially reduces the problem to a purely Riemannian geometry problem with the momentum information appearing as an extra term that needs to be controlled by means of the energy condition. We recall the reference isometric embedding used in the
Wang--Yau construction.  Let $(\Sigma,\sigma)$ be a smooth
Riemannian $2$--sphere and let
\[
T_0\in \mathbb R^{1,3}
\]
be a fixed future-directed unit timelike vector.  We write
$\langle\cdot,\cdot\rangle_\eta$ for the Minkowski inner product,
with signature $(-,+,+,+)$.

\begin{theorem}[Wang--Yau]
\label{thm:WY-reference-embedding}
Let $\chi\in C^\infty(\Sigma)$ satisfy
$\int_\Sigma \chi\,d\mu_\sigma=0,$
and let $\tau$ be the normalized solution of
\begin{equation}
\label{eq:tau-Poisson}
\Delta_\sigma\tau=\chi,
\end{equation}
Assume that
\begin{equation}
\label{eq:WY-convexity-condition}
K_\sigma
+
\frac{\det_\sigma(\nabla^2_\sigma\tau)}
     {1+|\nabla\tau|_\sigma^2}
>0
\qquad\text{on }\Sigma.
\end{equation}
Then there exists a spacelike isometric embedding
\[
i_0:(\Sigma,\sigma)\longrightarrow\mathbb R^{1,3}
\]
such that
\begin{equation}
\label{eq:WY-time-function}
\tau=-\langle i_0,T_0\rangle_\eta.
\end{equation}
If $\mathbf H_0$ denotes the mean-curvature vector of
$i_0(\Sigma)$, with the convention
$\mathbf H_0=\Delta_\sigma i_0$, then
\begin{equation}
\label{eq:H0-T0-prescribed}
\langle\mathbf H_0,T_0\rangle_\eta=-\chi.
\end{equation}
The embedding is unique up to an ambient isometry preserving $T_0$.
\end{theorem}

\noindent The reduction underlying Theorem~\ref{thm:WY-reference-embedding} is
particularly simple.  Let
\[
T_0^\perp
:=
\{V\in\mathbb R^{1,3}:
  \langle V,T_0\rangle_\eta=0\}
\simeq\mathbb R^3
\]
and define
\begin{equation}
\label{eq:reference-projection}
\widehat i_0
:=
i_0-\tau T_0.
\end{equation}
By \eqref{eq:WY-time-function},
$\widehat i_0(\Sigma)\subset T_0^\perp$.  Its induced metric is
\begin{equation}
\label{eq:hat-sigma}
\widehat\sigma
=
\sigma+d\tau\otimes d\tau.
\end{equation}
The Gaussian curvature of $\widehat\sigma$ is
\begin{equation}
\label{eq:hat-K}
\widehat K
=
\frac{1}{1+|\nabla\tau|_\sigma^2}
\left(
K_\sigma
+
\frac{\det_\sigma(\nabla^2_\sigma\tau)}
     {1+|\nabla\tau|_\sigma^2}
\right).
\end{equation}
Hence \eqref{eq:WY-convexity-condition} implies
$\widehat K>0$.  The Weyl--Nirenberg--Pogorelov theorem therefore
provides an isometric embedding
\[
\widehat i_0:
(\Sigma,\widehat\sigma)
\longrightarrow T_0^\perp\simeq\mathbb R^3,
\]
unique up to Euclidean rigid motions.  The corresponding Minkowski
embedding is recovered by
\[
i_0=\widehat i_0+\tau T_0.
\]
Indeed, for $V,W\in T\Sigma$,
\[
\begin{aligned}
\langle di_0(V),di_0(W)\rangle_\eta
&=
\langle d\widehat i_0(V),d\widehat i_0(W)\rangle_\eta
-d\tau(V)d\tau(W)\\
&=
\widehat\sigma(V,W)-d\tau(V)d\tau(W)
=
\sigma(V,W).
\end{aligned}
\]
Finally,
\[
\langle\mathbf H_0,T_0\rangle_\eta
=
\Delta_\sigma\langle i_0,T_0\rangle_\eta
=
-\Delta_\sigma\tau
=
-\chi,
\]
which gives \eqref{eq:H0-T0-prescribed}.

\noindent We now introduce the reference Hamiltonian density in the form needed
below.  Let $(e_3)_0$ be a spacelike unit normal to
$i_0(\Sigma)$ and let $(e_4)_0$ be the future-directed timelike unit
normal completing it to an oriented orthonormal frame of the normal
bundle.  Define
\begin{equation}
\label{eq:reference-normal-connection}
\alpha_{(e_3)_0}(Y)
:=
\left\langle
\nabla^\eta_Y(e_3)_0,(e_4)_0
\right\rangle_\eta,
\qquad
Y\in T\Sigma.
\end{equation}
The generalized reference mean-curvature density is
\begin{equation}
\label{eq:reference-generalized-H}
h_0(\tau,(e_3)_0)
:=
-\sqrt{1+|\nabla\tau|_\sigma^2}\,
 \langle\mathbf H_0,(e_3)_0\rangle_\eta
-\alpha_{(e_3)_0}(\nabla\tau).
\end{equation}

\noindent Similarly, for the physical embedding
$i:\Sigma\hookrightarrow(\mathcal M,\mathbf g)$, let
$\mathbf H$ denote its spacetime mean-curvature vector and let
$\{e_3,e_4\}$ be the corresponding Wang--Yau canonical normal frame.
We set
\begin{equation}
\label{eq:physical-generalized-H}
h(\tau,e_3)
:=
-\sqrt{1+|\nabla\tau|_\sigma^2}\,
 \mathbf g(\mathbf H,e_3)
-\alpha_{e_3}(\nabla\tau).
\end{equation}
For a fixed admissible time function $\tau$, the reduced Wang--Yau
energy is
\begin{equation}
\label{eq:reduced-WY-energy}
8\pi E_{\rm WY}(\Sigma,\tau)
=
\int_\Sigma
\left(
h_0(\tau,(e_3)_0)-h(\tau,e_3)
\right)d\mu_\sigma.
\end{equation}
The Wang--Yau mass is obtained by minimizing over the admissible class,
\begin{equation}
\label{eq:WY-mass-infimum}
M_{\rm WY}(\Sigma)
=
\inf_{\tau\in\mathcal A(\Sigma)}
E_{\rm WY}(\Sigma,\tau).
\end{equation}
Thus a uniform estimate
\[
E_{\rm WY}(\Sigma,\tau)\geq0
\qquad
\text{for all }\tau\in\mathcal A(\Sigma)
\]
immediately implies
\[
M_{\rm WY}(\Sigma)\geq0.
\]

\noindent The key point on the reference side is that the apparently Lorentzian
quantity in \eqref{eq:reference-generalized-H} is exactly the total
mean curvature of the Euclidean projection
$\widehat i_0(\Sigma)\subset T_0^\perp$.

\begin{proposition}[Reduction of the reference Hamiltonian]
\label{prop:reference-Hamiltonian-reduction}
Let
\[
i_0:(\Sigma,\sigma)\longrightarrow\mathbb R^{1,3}
\]
be an admissible Wang--Yau reference embedding, let
\[
\tau=-\langle i_0,T_0\rangle_\eta,
\qquad
\widehat\sigma=\sigma+d\tau\otimes d\tau,
\]
and let
\[
\widehat i_0=i_0-\tau T_0:
(\Sigma,\widehat\sigma)\longrightarrow T_0^\perp
\]
be its orthogonal projection. Let $\widehat e_3$ denote the outward Euclidean unit normal to
$\widehat i_0(\Sigma)\subset T_0^\perp$, extended off $T_0^\perp$
by parallel translation along the integral curves of $T_0$, and let
$\widehat k_0$ be the Euclidean mean curvature of
$\widehat i_0(\Sigma)$.  Then $\widehat e_3$ is also a spacelike unit
normal to $i_0(\Sigma)$, and
\begin{equation}
\label{eq:reference-pointwise-reduction}
\widehat k_0
=
-\langle\mathbf H_0,\widehat e_3\rangle_\eta
-
\frac{1}{\sqrt{1+|\nabla\tau|_\sigma^2}}\,
\alpha_{\widehat e_3}(\nabla\tau).
\end{equation}
Consequently,
\begin{equation}
\label{eq:reference-integrated-reduction}
\int_{\widehat i_0(\Sigma)}
\widehat k_0\,d\mu_{\widehat\sigma}
=
\int_\Sigma
\left(
-\sqrt{1+|\nabla\tau|_\sigma^2}\,
 \langle\mathbf H_0,\widehat e_3\rangle_\eta
-\alpha_{\widehat e_3}(\nabla\tau)
\right)d\mu_\sigma.
\end{equation}
\end{proposition}

\begin{proof}
Set
\[
a:=\sqrt{1+|\nabla\tau|_\sigma^2}.
\]
Since
\[
\tau=-\langle i_0,T_0\rangle_\eta,
\]
for every $Y\in T\Sigma$,
\[
d\tau(Y)
=
-\langle Y,T_0\rangle_\eta.
\]
Hence the tangential component of $T_0$ along $i_0(\Sigma)$ is
\[
T_0^\top=-\nabla\tau.
\]
It follows that
\begin{equation}
\label{eq:T0-decomposition}
T_0=-\nabla\tau+a\,\widehat e_4,
\end{equation}
where
\[
\widehat e_4
:=
\frac{T_0+\nabla\tau}{a}
\]
is the future-directed unit timelike normal to $i_0(\Sigma)$
orthogonal to $\widehat e_3$. The vector field $\widehat e_3$ is normal to $i_0(\Sigma)$ because,
for every $Y\in T\Sigma$,
\[
di_0(Y)
=
d\widehat i_0(Y)+d\tau(Y)T_0,
\]
while
\[
\langle\widehat e_3,d\widehat i_0(Y)\rangle_\eta
=
\langle\widehat e_3,T_0\rangle_\eta
=
0.
\]
Let $\{e_1,e_2\}$ be a local $\sigma$-orthonormal tangent frame of
$i_0(\Sigma)$.  Since $\widehat e_3$ is unit spacelike,
\[
\left\langle
\nabla^\eta_{\widehat e_3}\widehat e_3,
\widehat e_3
\right\rangle_\eta=0.
\]
Moreover, by construction,
\[
\nabla^\eta_{T_0}\widehat e_3=0.
\]
The ambient Lorentzian divergence of $\widehat e_3$ may therefore be
computed either in the frame adapted to the Euclidean hypersurface
$T_0^\perp$ or in the orthonormal frame
$\{e_1,e_2,\widehat e_3,\widehat e_4\}$.  This gives
\begin{equation}
\label{eq:k0-divergence}
\widehat k_0
=
\sum_{a=1}^2
\left\langle
\nabla^\eta_{e_a}\widehat e_3,e_a
\right\rangle_\eta
-
\left\langle
\nabla^\eta_{\widehat e_4}\widehat e_3,\widehat e_4
\right\rangle_\eta.
\end{equation}
With the convention
\[
\mathbf H_0
=
\sum_{a=1}^2
\left(\nabla^\eta_{e_a}e_a\right)^\perp,
\]
metric compatibility yields
\[
\sum_{a=1}^2
\left\langle
\nabla^\eta_{e_a}\widehat e_3,e_a
\right\rangle_\eta
=
-\langle\mathbf H_0,\widehat e_3\rangle_\eta.
\]
Thus
\begin{equation}
\label{eq:k0-intermediate}
\widehat k_0
=
-\langle\mathbf H_0,\widehat e_3\rangle_\eta
-
\left\langle
\nabla^\eta_{\widehat e_4}\widehat e_3,\widehat e_4
\right\rangle_\eta.
\end{equation}
Using \eqref{eq:T0-decomposition} and
$\nabla^\eta_{T_0}\widehat e_3=0$, we obtain
\[
0
=
\left\langle
\nabla^\eta_{T_0}\widehat e_3,\widehat e_4
\right\rangle_\eta
=
-
\left\langle
\nabla^\eta_{\nabla\tau}\widehat e_3,\widehat e_4
\right\rangle_\eta
+
a
\left\langle
\nabla^\eta_{\widehat e_4}\widehat e_3,\widehat e_4
\right\rangle_\eta.
\]
Therefore
\begin{equation}
\label{eq:e4-derivative}
\left\langle
\nabla^\eta_{\widehat e_4}\widehat e_3,\widehat e_4
\right\rangle_\eta
=
\frac{1}{a}
\alpha_{\widehat e_3}(\nabla\tau).
\end{equation}
Substituting \eqref{eq:e4-derivative} into
\eqref{eq:k0-intermediate} proves
\eqref{eq:reference-pointwise-reduction}. Finally, the matrix determinant lemma applied to
\[
\widehat\sigma=\sigma+d\tau\otimes d\tau
\]
gives
\begin{equation}
\label{eq:area-form-reference}
d\mu_{\widehat\sigma}
=
\sqrt{1+|\nabla\tau|_\sigma^2}\,d\mu_\sigma
=
a\,d\mu_\sigma.
\end{equation}
Multiplying \eqref{eq:reference-pointwise-reduction} by
$d\mu_{\widehat\sigma}$ and using
\eqref{eq:area-form-reference} yields
\eqref{eq:reference-integrated-reduction}.
\end{proof}
\begin{remark} Proposition~\ref{prop:reference-Hamiltonian-reduction} completely
eliminates the Minkowskian reference term: the reference Hamiltonian is
precisely the total Euclidean mean curvature of the projected surface
\[
(\widehat\Sigma,\widehat\sigma)
\simeq
\widehat i_0(\Sigma)\subset T_0^\perp\simeq\mathbb R^3.
\]
It remains to perform the analogous reduction of the physical
Hamiltonian.  This is achieved by solving Jang's equation with boundary
value $\tau$ and expressing the physical generalized mean curvature in
terms of the boundary geometry of the resulting Jang graph.
\end{remark}

\noindent We now turn to the physical Hamiltonian.  In contrast with the reference
term, the physical contribution depends not only on the intrinsic
geometry of $\Sigma$, but also on the momentum data of the ambient
initial hypersurface.  The appropriate analogue of the Euclidean
projection considered above is the Jang deformation.

\noindent Let
\[
(\Omega,g,P),\qquad \partial\Omega=\Sigma,
\]
be the physical initial data, where $P$ is the second fundamental form
of $\Omega$ in the spacetime.  Consider the Riemannian product
\[
(\Omega\times\mathbb R,g+dt^2),
\]
and extend $P$ trivially in the $\mathbb R$-direction:
\begin{equation}
\label{eq:P-extension}
P(\partial_t,\cdot)=0,
\qquad
P|_{T\Omega\times T\Omega}=P.
\end{equation}

\noindent For a function $f:\Omega\rightarrow\mathbb R$, let
\[
\widehat\Omega
:=
\{(x,f(x)):x\in\Omega\}
\subset\Omega\times\mathbb R
\]
be its graph.  The induced metric is
\begin{equation}
\label{eq:Jang-induced-metric}
\widehat g_{ij}
=
g_{ij}+f_i f_j,
\qquad
\widehat g^{ij}
=
g^{ij}
-
\frac{f^if^j}{1+|\nabla f|_g^2}.
\end{equation}
Writing
\[
W:=\sqrt{1+|\nabla f|_g^2},
\]
the downward-pointing unit normal to $\widehat\Omega$ is
\begin{equation}
\label{eq:Jang-downward-normal}
\widetilde e_4
=
\frac{1}{W}
\left(
f^i\partial_i-\partial_t
\right).
\end{equation}
With this convention, the second fundamental form of the graph is
\[
\widehat A_{ij}
=
\frac{\nabla_i\nabla_jf}{W}.
\]

\noindent Jang's equation requires the mean curvature of the graph to agree with
the trace of $P$ over $T\widehat\Omega$:
\begin{equation}
\label{eq:Jang-geometric}
\operatorname{tr}_{\widehat g}(\widehat A-P)=0.
\end{equation}
Equivalently,
\begin{equation}
\label{eq:Jang-coordinate}
\left(
g^{ij}
-
\frac{f^if^j}{1+|\nabla f|_g^2}
\right)
\left(
\frac{\nabla_i\nabla_jf}
     {\sqrt{1+|\nabla f|_g^2}}
-P_{ij}
\right)
=0.
\end{equation}
This is a quasilinear elliptic equation as long as the graph remains
regular. For a prescribed function $\tau\in C^\infty(\Sigma)$, we consider the
Dirichlet problem
\begin{equation}
\label{eq:Jang-Dirichlet}
\left\{
\begin{aligned}
\operatorname{tr}_{\widehat g}(\widehat A-P)&=0
&&\text{in }\Omega,\\
f&=\tau
&&\text{on }\Sigma.
\end{aligned}
\right.
\end{equation}
The boundary of the resulting graph is
\[
\widehat\Sigma
=
\{(x,\tau(x)):x\in\Sigma\}
=
\partial\widehat\Omega,
\]
and its induced metric is
\begin{equation}
\label{eq:Jang-boundary-metric}
\widehat\sigma
=
\sigma+d\tau\otimes d\tau.
\end{equation}
This is precisely the metric induced on the Euclidean projection of
the Wang--Yau reference embedding.  Consequently, the reference and
physical Hamiltonians are naturally expressed on the same Riemannian
surface $(\widehat\Sigma,\widehat\sigma)$.

\noindent The Dirichlet problem for Jang's equation may fail through blow-up at
apparent horizons.  In the regime relevant here, this obstruction is
excluded by the Wang--Yau admissibility hypotheses.  The boundary
estimate needed in the continuity argument is the following.

\begin{theorem}[Wang--Yau]
\label{thm:Jang-boundary-gradient}
Let $f$ solve \eqref{eq:Jang-coordinate} with
$f|_\Sigma=\tau$.  If
\begin{equation}
\label{eq:Jang-boundary-barrier}
\widehat k>
\left|
\operatorname{tr}_{\widehat\Sigma}P
\right|
\qquad\text{on }\widehat\Sigma,
\end{equation}
where $\widehat k$ denotes the mean curvature of
$\widehat\Sigma\subset\widehat\Omega$ with respect to the outward
normal, then the boundary normal derivative of $f$ is uniformly
bounded.
\end{theorem}

\noindent Together with the interior estimates for Jang's equation and the
absence of an apparent-horizon obstruction, this boundary estimate
yields solvability of \eqref{eq:Jang-Dirichlet}; see~\cite{yau} and
\cite{schoen1979proof,schoen1981proof}.

We record the admissibility condition in the form used below.

\begin{definition}[Admissible time function]
\label{admi}
Let $i:\Sigma\hookrightarrow(\mathcal M,\mathbf g)$ be a spacelike
embedding.  A function $\tau\in C^\infty(\Sigma)$ is called
\emph{admissible} if the following conditions hold:
\begin{enumerate}
\item
the projected metric
\[
\widehat\sigma
=
\sigma+d\tau\otimes d\tau
\]
has strictly positive Gaussian curvature, equivalently,
\begin{equation}
\label{eq:admissibility-convexity}
K_\sigma
+
\frac{\det_\sigma(\nabla_\sigma^2\tau)}
     {1+|\nabla\tau|_\sigma^2}
>0;
\end{equation}

\item
$\Sigma$ bounds a compact spacelike hypersurface $\Omega$ for which
the Dirichlet problem \eqref{eq:Jang-Dirichlet} admits a smooth
solution;

\item
the generalized physical mean curvature in the canonical normal gauge
determined by the Jang solution satisfies
\begin{equation}
\label{eq:admissibility-H}
h(\Sigma,i,\tau,e_3')>0.
\end{equation}
\end{enumerate}
\end{definition}

\noindent We finally fix the frames used in the physical reduction.  Along
$\widehat\Sigma$, let
\[
\{\widetilde e_1,\widetilde e_2\}
\subset T\widehat\Sigma
\]
be an orthonormal tangent frame with respect to $\widehat\sigma$.
Let $\widetilde e_3$ be the outward unit normal to
$\widehat\Sigma$ within $(\widehat\Omega,\widehat g)$, and let
$\widetilde e_4$ be the downward unit normal to
$\widehat\Omega\subset\Omega\times\mathbb R$.  Thus
\[
\{\widetilde e_1,\widetilde e_2,
  \widetilde e_3,\widetilde e_4\}
\]
is an orthonormal frame of $T(\Omega\times\mathbb R)$ along
$\widehat\Sigma$.

\noindent Likewise, if $\{e_1,e_2\}$ is an orthonormal frame tangent to $\Sigma$
and $e_3$ is the outward unit normal to $\Sigma$ in $(\Omega,g)$, then
\[
\{e_1,e_2,e_3,v\},
\qquad
v:=-\partial_t,
\]
is an orthonormal frame of the product metric along $\Sigma\times
\mathbb R$.  All fields are extended in the $\mathbb R$-direction by
parallel translation. With these conventions, the physical Wang--Yau Hamiltonian can be
expressed entirely in terms of the boundary geometry of the Jang graph
together with the tensor $P$.  This is the content of the next
proposition.

\begin{proposition}
\label{mean2}
Let $(\Omega,g,P)$ be a spacelike initial data set with
$\partial\Omega=\Sigma$, and let
$f\in C^\infty(\Omega)$ solve Jang's equation with boundary value
\[
f|_\Sigma=\tau.
\]
Let $(\widehat\Omega,\widehat g)$ be the graph of $f$ in
$(\Omega\times\mathbb R,g+dt^2)$ and
$\widehat\Sigma=\partial\widehat\Omega$ the graph of $\tau$.
Denote by $\widetilde e_4$ the downward unit normal to
$\widehat\Omega$ and by $\widetilde e_3$ the outward unit normal to
$\widehat\Sigma$ in $\widehat\Omega$.  Let $\widetilde k$ be the
corresponding mean curvature of $\widehat\Sigma$.

\noindent Let $e_3$ be the outward unit normal to $\Sigma$ in $\Omega$, and
let $e_4$ be the future-directed timelike unit normal to $\Omega$ in
the physical spacetime.  Set
\[
f_3:=e_3(f)|_\Sigma,
\qquad
a:=\sqrt{1+|\nabla\tau|_\sigma^2},
\qquad
W:=\sqrt{1+|\nabla f|_g^2}\big|_\Sigma
   =\sqrt{a^2+f_3^2}.
\]
Define $\varphi$ by
\begin{equation}
\label{eq:boost-phi-Jang}
\sinh\varphi=-\frac{f_3}{a},
\qquad
\cosh\varphi=\frac{W}{a},
\end{equation}
and introduce the boosted normal frame
\begin{equation}
\label{eq:boosted-physical-frame}
e_3'
=
\cosh\varphi\,e_3+\sinh\varphi\,e_4,
\qquad
e_4'
=
\sinh\varphi\,e_3+\cosh\varphi\,e_4.
\end{equation}
Then
\begin{equation}
\label{eq:physical-Jang-pointwise}
\widetilde k
-
\left\langle
\widetilde\nabla_{\widetilde e_4}\widetilde e_4,
\widetilde e_3
\right\rangle
+
P(\widetilde e_4,\widetilde e_3)
=
-\langle\mathbf H,e_3'\rangle
-\frac{1}{a}\alpha_{e_3'}(\nabla\tau).
\end{equation}
Consequently,
\begin{equation}
\label{eq:physical-Jang-integrated}
\begin{aligned}
&\int_{\widehat\Sigma}
\left(
\widetilde k
-
\left\langle
\widetilde\nabla_{\widetilde e_4}\widetilde e_4,
\widetilde e_3
\right\rangle
+
P(\widetilde e_4,\widetilde e_3)
\right)d\mu_{\widehat\sigma}
\\
&\hspace{2cm}
=
\int_\Sigma
\left(
-\sqrt{1+|\nabla\tau|_\sigma^2}\,
 \langle\mathbf H,e_3'\rangle
-\alpha_{e_3'}(\nabla\tau)
\right)d\mu_\sigma .
\end{aligned}
\end{equation}
\end{proposition}

\begin{proof}
We record the boundary geometry of the Jang graph.  Along $\Sigma$,
write
\[
\nabla f=\nabla\tau+f_3e_3.
\]
With $\partial_t$ denoting the unit vector in the $\mathbb R$-factor,
the downward unit normal to the graph is
\begin{equation}
\label{eq:tilde-e4-explicit}
\widetilde e_4
=
\frac{1}{W}
\left(
\nabla\tau+f_3e_3-\partial_t
\right).
\end{equation}
The outward unit normal to $\widehat\Sigma$ within the graph is
\begin{equation}
\label{eq:tilde-e3-explicit}
\widetilde e_3
=
\frac{a}{W}e_3
-
\frac{f_3}{aW}\nabla\tau
+
\frac{f_3}{aW}\partial_t.
\end{equation}
In particular,
\begin{equation}
\label{eq:graph-inner-products}
\langle e_3,\widetilde e_3\rangle
=
\frac{a}{W},
\qquad
\langle e_3,\widetilde e_4\rangle
=
\frac{f_3}{W}.
\end{equation}
Let
\[
k:=\sum_{A=1}^2
\langle\nabla_{e_A}e_3,e_A\rangle,
\qquad
p:=\operatorname{tr}_\Sigma P
=\sum_{A=1}^2P(e_A,e_A).
\]
Using \eqref{eq:tilde-e4-explicit}--\eqref{eq:graph-inner-products},
decomposing the traces with respect to the two adapted orthonormal
frames, and applying Jang's equation
\[
\sum_{i=1}^3
\left\langle
\widetilde\nabla_{\widetilde e_i}\widetilde e_4,
\widetilde e_i
\right\rangle
=
\sum_{i=1}^3P(\widetilde e_i,\widetilde e_i),
\]
one obtains
\begin{equation}
\label{eq:key-Jang-boundary-identity}
\begin{aligned}
&a\left(
\widetilde k
-
\left\langle
\widetilde\nabla_{\widetilde e_4}\widetilde e_4,
\widetilde e_3
\right\rangle
+
P(\widetilde e_4,\widetilde e_3)
\right)
\\
&\qquad
=
Wk-f_3p
-\alpha_{e_3}(\nabla\tau)
+\nabla\tau(\varphi).
\end{aligned}
\end{equation}
Here we have used
\[
P(e_3,\nabla\tau)
=
-\alpha_{e_3}(\nabla\tau),
\]
which follows from the definition of $P$ and metric compatibility.

\noindent It remains to identify the right-hand side of
\eqref{eq:key-Jang-boundary-identity} with the physical Hamiltonian.
With the convention for the spacetime mean-curvature vector used
throughout,
\begin{equation}
\label{eq:H-decomposition}
\mathbf H=-k\,e_3+p\,e_4.
\end{equation}
Therefore, by \eqref{eq:boost-phi-Jang},
\begin{equation}
\label{eq:H-boost-computation}
-a\langle\mathbf H,e_3'\rangle
=
a\bigl(k\cosh\varphi+p\sinh\varphi\bigr)
=
Wk-f_3p.
\end{equation}
Under the Lorentz boost \eqref{eq:boosted-physical-frame}, the normal
connection transforms as
\begin{equation}
\label{eq:alpha-boost}
\alpha_{e_3'}
=
\alpha_{e_3}-d\varphi.
\end{equation}
Hence
\[
-\alpha_{e_3'}(\nabla\tau)
=
-\alpha_{e_3}(\nabla\tau)
+\nabla\tau(\varphi).
\]
Combining this with \eqref{eq:key-Jang-boundary-identity} and
\eqref{eq:H-boost-computation} proves
\eqref{eq:physical-Jang-pointwise}.
Finally,
\[
\widehat\sigma
=
\sigma+d\tau\otimes d\tau,
\qquad
d\mu_{\widehat\sigma}
=
a\,d\mu_\sigma.
\]
Multiplying \eqref{eq:physical-Jang-pointwise} by
$d\mu_{\widehat\sigma}$ and integrating gives
\eqref{eq:physical-Jang-integrated}.
\end{proof}
\section{Proof of the main Theorems}
\label{5}
\noindent Now we have the basic ingredients for the proof of our main result. Firstly, we define the necessary domain and entities.
Let $\Omega_0\Subset{\mathbb R}^3$ be a smooth bounded domain with
connected strictly convex boundary
\[
        \Sigma_0:=\partial\Omega_0.
\]
We denote by $\delta$ the Euclidean metric, by
\[
        \gamma_0:=\delta|_{T\Sigma_0}
\]
the induced metric, by $\nu_0$ the outward Euclidean unit normal, by
$A_0$ the second fundamental form, and by
\[
        H_0:=\operatorname{tr}_{\gamma_0}A_0
\]
the Euclidean mean curvature.  Our sign convention is chosen so that
\[
        A_0>0,\qquad H_0>0
\]
for a strictly convex Euclidean boundary. In this section, we first recall the scalar-flat reduction of the Wang--Yau energy. Essentially following remarkable proposition accomplishes the fact that the positivity of the Wang-Yau energy reduces to the positivity of the Brown-- York energy for the conformally scalar-flat Jang fill-in data. This is essentially Theorem \ref{1}. The following proposition and its proof essentially accomplish it.

\begin{proposition}
\label{prop:WY-scalar-flat-reduction}
Let $\tau$ be an admissible Wang--Yau time function and let
$(\widehat\Omega,\widehat g)$ be the corresponding Jang graph, with
boundary $\widehat\Sigma$.  Let $\nu$ denote the outward
$\widehat g$-unit normal to $\widehat\Sigma$ and let $\widehat H$ denote
its mean curvature.  Define
\[
        X^\flat
        :=
        \left\langle
        \widetilde\nabla_{\widetilde e_4}\widetilde e_4,\cdot
        \right\rangle
        -
        P(\widetilde e_4,\cdot).
\]
Then the dominant energy condition implies
\begin{equation}
\label{eq:WY-Jang-R}
        R_{\widehat g}
        \geq
        2|X|_{\widehat g}^2
        -
        2\operatorname{div}_{\widehat g}X .
\end{equation}
There exists a unique positive solution $U$ of
\begin{equation}
\label{eq:WY-conformal-equation}
\left\{
\begin{aligned}
        -8\Delta_{\widehat g}U
        +R_{\widehat g}U&=0
        &&\text{in }\widehat\Omega,\\
        U&=1
        &&\text{on }\widehat\Sigma .
\end{aligned}
\right.
\end{equation}
The conformally related metric
$g_{\rm sf}:=U^4\widehat g$
satisfies
\begin{equation}
\label{eq:WY-sf-properties}
        R_{g_{\rm sf}}=0,
        \qquad
        g_{\rm sf}|_{T\widehat\Sigma}
        =
        \widehat\sigma .
\end{equation}
Then
\begin{equation}
\label{eq:WY-ge-BY}
        E_{\rm WY}(\Sigma,\tau)
        \geq
        m_{\rm BY}
        (\widehat\Sigma;\widehat\Omega,g_{\rm sf}).
\end{equation}

\end{proposition}

\begin{proof}
Throughout the proof, $\nu$ denotes the outward
$\widehat g$--unit normal to $\widehat\Sigma$, and we use the
conventions
\[
\widehat\Delta
=
\operatorname{div}_{\widehat g}\widehat\nabla,
\qquad
\widehat A(Y,Z)
=
\widehat g(\widehat\nabla_Y\nu,Z),
\qquad
\widehat H
=
\operatorname{tr}_{\widehat\sigma}\widehat A .
\]
Recall that
\begin{equation}
\label{eq:Jang-R-lower}
R_{\widehat g}
\geq
2|X|_{\widehat g}^{2}
-
2\operatorname{div}_{\widehat g}X,
\end{equation}
and
\begin{equation}
\label{eq:Jang-generalized-H}
\mathcal H
=
\widehat H-\langle X,\nu\rangle_{\widehat g}.
\end{equation}
First, we prove the existence and the uniqueness of the solution of the Dirichlet problem 
\begin{equation}
\begin{aligned}
        -8\Delta_{\widehat g}U
        +R_{\widehat g}U&=0
        &&\text{in }\widehat\Omega,\\
        U&=1
        &&\text{on }\widehat\Sigma .
\end{aligned}
\end{equation}
This follows using a direct integration by parts argument. For $v\in H^1_0(\widehat\Omega)$, the Jang scalar-curvature
inequality
\[
R_{\widehat g}
\geq
2|X|_{\widehat g}^2
-
2\operatorname{div}_{\widehat g}X
\]
gives
\begin{align}
&\int_{\widehat\Omega}
\left(
8|\widehat\nabla v|_{\widehat g}^2
+
R_{\widehat g}v^2
\right)dV_{\widehat g}
\nonumber\\
&\qquad\geq
8\int_{\widehat\Omega}
|\widehat\nabla v|_{\widehat g}^2\,dV_{\widehat g}
+
2\int_{\widehat\Omega}
|X|_{\widehat g}^2v^2\,dV_{\widehat g}
-
2\int_{\widehat\Omega}
v^2\operatorname{div}_{\widehat g}X\,dV_{\widehat g}.
\label{eq:coercivity-1}
\end{align}
Since
\[
\operatorname{div}_{\widehat g}(v^2X)
=
v^2\operatorname{div}_{\widehat g}X
+
2v\langle X,\widehat\nabla v\rangle_{\widehat g},
\]
the divergence theorem yields
\begin{align}
\int_{\widehat\Omega}
v^2\operatorname{div}_{\widehat g}X\,dV_{\widehat g}
&=
\int_{\partial\widehat\Omega}
v^2\langle X,\nu\rangle_{\widehat g}\,
dA_{\widehat g}
-
2\int_{\widehat\Omega}
v\langle X,\widehat\nabla v\rangle_{\widehat g}\,
dV_{\widehat g}.
\end{align}
Because $v\in H^1_0(\widehat\Omega)$, its trace vanishes on
$\partial\widehat\Omega$, and hence
\[
\int_{\widehat\Omega}
v^2\operatorname{div}_{\widehat g}X\,dV_{\widehat g}
=
-2\int_{\widehat\Omega}
v\langle X,\widehat\nabla v\rangle_{\widehat g}\,
dV_{\widehat g}.
\]
Therefore \eqref{eq:coercivity-1} becomes
\begin{align}
&\int_{\widehat\Omega}
\left(
8|\widehat\nabla v|_{\widehat g}^2
+
R_{\widehat g}v^2
\right)dV_{\widehat g}
\nonumber\\
&\qquad\geq
\int_{\widehat\Omega}
\left(
8|\widehat\nabla v|_{\widehat g}^2
+
2v^2|X|_{\widehat g}^2
+
4v\langle X,\widehat\nabla v\rangle_{\widehat g}
\right)dV_{\widehat g}.
\label{eq:coercivity-2}
\end{align}
Now
\begin{align}
&8|\widehat\nabla v|_{\widehat g}^2
+
2v^2|X|_{\widehat g}^2
+
4v\langle X,\widehat\nabla v\rangle_{\widehat g}
\nonumber\\
&\qquad=
6|\widehat\nabla v|_{\widehat g}^2
+
2\left(
|\widehat\nabla v|_{\widehat g}^2
+
v^2|X|_{\widehat g}^2
+
2v\langle X,\widehat\nabla v\rangle_{\widehat g}
\right)
\nonumber\\
&\qquad=
6|\widehat\nabla v|_{\widehat g}^2
+
2|\widehat\nabla v+vX|_{\widehat g}^2.
\end{align}
Consequently,
\begin{equation}
\label{eq:conformal-coercivity}
\begin{aligned}
\int_{\widehat\Omega}
\left(
8|\widehat\nabla v|_{\widehat g}^2
+
R_{\widehat g}v^2
\right)dV_{\widehat g}
\geq{}&
6\int_{\widehat\Omega}
|\widehat\nabla v|_{\widehat g}^2\,dV_{\widehat g}
+
2\int_{\widehat\Omega}
|\widehat\nabla v+vX|_{\widehat g}^2\,dV_{\widehat g}.
\end{aligned}
\end{equation}
By the coercivity estimate established above, the Dirichlet realization
of
\[
L_{\widehat g}
:=
-8\widehat\Delta+R_{\widehat g}
\]
is strictly positive.  Hence the boundary-value problem
\begin{equation}
\label{eq:WY-conformal-equation}
\left\{
\begin{aligned}
-8\widehat\Delta U+R_{\widehat g}U&=0
&&\text{in }\widehat\Omega,\\
U&=1
&&\text{on }\widehat\Sigma
\end{aligned}
\right.
\end{equation}
has a unique solution.  Finally, for the positivity, let
\[
U_-:=\max\{-U,0\}\in H^1_0(\widehat\Omega).
\]
Testing $L_{\widehat g}U=0$ against $U_-$ gives
\[
0
=
\int_{\widehat\Omega}
\left(
8\langle\widehat\nabla U,\widehat\nabla U_-\rangle
+
R_{\widehat g}UU_-
\right)dV_{\widehat g}.
\]
On $\{U<0\}$,
\[
U=-U_-,
\qquad
\widehat\nabla U=-\widehat\nabla U_-,
\]
and therefore
\[
0
=
-
\int_{\widehat\Omega}
\left(
8|\widehat\nabla U_-|^2
+
R_{\widehat g}U_-^2
\right)dV_{\widehat g}.
\]
Hence, by \eqref{eq:conformal-coercivity},
\[
0
\geq
6\int_{\widehat\Omega}
|\widehat\nabla U_-|^2\,dV_{\widehat g},
\]
so
\[
U_-\equiv0.
\]
Thus $U\geq0$.  Since $U=1$ on
$\partial\widehat\Omega$, $U\not\equiv0$, and the strong maximum
principle gives
\[
U>0
\qquad\text{in }\widehat\Omega.
\]
Set
\begin{equation}
\label{eq:gsf-definition}
g_{\rm sf}:=U^4\widehat g.
\end{equation}
We first verify scalar flatness explicitly.  Write
\[
g_{\rm sf}=e^{2\varphi}\widehat g,
\qquad
\varphi:=2\log U.
\]
In dimension three the scalar curvature under
$\widetilde g=e^{2\varphi}g$ transforms according to
\[
R_{\widetilde g}
=
e^{-2\varphi}
\left(
R_g-4\Delta_g\varphi-2|\nabla\varphi|_g^2
\right).
\]
Since
\[
\widehat\nabla\varphi
=
2U^{-1}\widehat\nabla U,
\qquad
|\widehat\nabla\varphi|_{\widehat g}^2
=
4U^{-2}|\widehat\nabla U|_{\widehat g}^2,
\]
and
\[
\widehat\Delta\varphi
=
2\widehat\Delta(\log U)
=
2U^{-1}\widehat\Delta U
-
2U^{-2}|\widehat\nabla U|_{\widehat g}^2,
\]
we obtain
\begin{align}
R_{g_{\rm sf}}
&=
U^{-4}
\left[
R_{\widehat g}
-
8U^{-1}\widehat\Delta U
+
8U^{-2}|\widehat\nabla U|^2
-
8U^{-2}|\widehat\nabla U|^2
\right]
\nonumber\\
&=
U^{-4}
\left(
R_{\widehat g}
-
8U^{-1}\widehat\Delta U
\right)
\nonumber\\
&=
U^{-5}
\left(
-8\widehat\Delta U
+
R_{\widehat g}U
\right)
=0.
\label{eq:scalar-curvature-conformal-explicit}
\end{align}
Thus
\[
R_{g_{\rm sf}}=0.
\]
Moreover,
\[
g_{\rm sf}|_{T\widehat\Sigma}
=
U^4\widehat g|_{T\widehat\Sigma}
=
\widehat\sigma,
\]
because $U=1$ on $\widehat\Sigma$.  Hence the boundary metric, and
therefore its Euclidean isometric embedding and the corresponding
reference mean curvature $k_0$, are unchanged.

\noindent We next compute the transformation of the boundary mean curvature. This is standard computation, but we provide the detailed comupations for self consistency.
More generally, let
\[
\widetilde g=e^{2\varphi}g.
\]
If $\nu$ is the $g$--unit outward normal, then
\[
\widetilde\nu=e^{-\varphi}\nu,
\]
since
\[
\widetilde g(\widetilde\nu,\widetilde\nu)
=
e^{2\varphi}
g(e^{-\varphi}\nu,e^{-\varphi}\nu)
=1.
\]
The Levi--Civita connections are related by
\begin{equation}
\label{eq:conformal-connection}
\widetilde\nabla_YZ
=
\nabla_YZ
+
Y(\varphi)Z
+
Z(\varphi)Y
-
g(Y,Z)\nabla\varphi.
\end{equation}
For $Y\in T\widehat\Sigma$,
\begin{align}
\widetilde\nabla_Y\widetilde\nu
&=
\widetilde\nabla_Y(e^{-\varphi}\nu)
\nonumber\\
&=
-e^{-\varphi}Y(\varphi)\nu
+
e^{-\varphi}\widetilde\nabla_Y\nu
\nonumber\\
&=
-e^{-\varphi}Y(\varphi)\nu
+
e^{-\varphi}
\left(
\nabla_Y\nu
+
Y(\varphi)\nu
+
\nu(\varphi)Y
-
g(Y,\nu)\nabla\varphi
\right).
\end{align}
Since $Y$ is tangential,
\[
g(Y,\nu)=0,
\]
and therefore
\begin{equation}
\label{eq:normal-connection-conformal}
\widetilde\nabla_Y\widetilde\nu
=
e^{-\varphi}
\left(
\nabla_Y\nu+\nu(\varphi)Y
\right).
\end{equation}
Consequently the second fundamental forms satisfy, for
$Y,Z\in T\widehat\Sigma$,
\begin{align}
\widetilde A(Y,Z)
&=
\widetilde g
\left(
\widetilde\nabla_Y\widetilde\nu,Z
\right)
\nonumber\\
&=
e^{2\varphi}
g
\left(
e^{-\varphi}
\left(
\nabla_Y\nu+\nu(\varphi)Y
\right),
Z
\right)
\nonumber\\
&=
e^\varphi
\left(
A(Y,Z)
+
\nu(\varphi)\gamma(Y,Z)
\right).
\label{eq:A-conformal}
\end{align}
Since
\[
\widetilde\gamma=e^{2\varphi}\gamma,
\qquad
\widetilde\gamma^{ab}
=
e^{-2\varphi}\gamma^{ab},
\]
and $\dim\widehat\Sigma=2$, tracing
\eqref{eq:A-conformal} gives
\begin{align}
\widetilde H
&=
\widetilde\gamma^{ab}\widetilde A_{ab}
\nonumber\\
&=
e^{-2\varphi}\gamma^{ab}
e^\varphi
\left(
A_{ab}+\nu(\varphi)\gamma_{ab}
\right)
\nonumber\\
&=
e^{-\varphi}
\left(
H+\nu(\varphi)\gamma^{ab}\gamma_{ab}
\right)
\nonumber\\
&=
e^{-\varphi}
\left(
H+2\nu(\varphi)
\right).
\label{eq:H-general-conformal}
\end{align}
For
\[
g_{\rm sf}=U^4\widehat g,
\qquad
\varphi=2\log U,
\]
we have
\[
e^{-\varphi}=U^{-2},
\qquad
\nu(\varphi)
=
2U^{-1}\partial_\nu U,
\]
and hence
\begin{equation}
\label{eq:Hsf-full}
H_{\rm sf}
=
U^{-2}
\left(
\widehat H
+
4U^{-1}\partial_\nu U
\right).
\end{equation}
Since $U=1$ on $\widehat\Sigma$,
\begin{equation}
\label{eq:Hsf-boundary}
H_{\rm sf}
=
\widehat H+4\partial_\nu U.
\end{equation}
We now derive the required boundary flux estimate.  Multiplying
\eqref{eq:WY-conformal-equation} by $U$ and integrating over
$\widehat\Omega$ gives
\[
0
=
-8\int_{\widehat\Omega}
U\widehat\Delta U\,dV_{\widehat g}
+
\int_{\widehat\Omega}
R_{\widehat g}U^2\,dV_{\widehat g}.
\]
Green's formula, with $\nu$ outward, yields
\[
\int_{\widehat\Omega}
U\widehat\Delta U\,dV_{\widehat g}
=
\int_{\widehat\Sigma}
U\partial_\nu U\,dA_{\widehat\sigma}
-
\int_{\widehat\Omega}
|\widehat\nabla U|^2\,dV_{\widehat g}.
\]
Since $U=1$ on $\widehat\Sigma$,
\begin{equation}
\label{eq:integrated-conformal-U}
8\int_{\widehat\Sigma}
\partial_\nu U\,dA_{\widehat\sigma}
=
\int_{\widehat\Omega}
\left(
8|\widehat\nabla U|^2
+
R_{\widehat g}U^2
\right)dV_{\widehat g}.
\end{equation}
Using \eqref{eq:Jang-R-lower},
\begin{align}
8\int_{\widehat\Sigma}\partial_\nu U\,dA_{\widehat\sigma}
\geq{}&
8\int_{\widehat\Omega}|\widehat\nabla U|^2\,dV_{\widehat g}
+
2\int_{\widehat\Omega}U^2|X|^2\,dV_{\widehat g}
\nonumber\\
&-
2\int_{\widehat\Omega}
U^2\operatorname{div}_{\widehat g}X\,dV_{\widehat g}.
\label{eq:flux-step1}
\end{align}
Now compute
\[
\operatorname{div}_{\widehat g}(U^2X)
=
U^2\operatorname{div}_{\widehat g}X
+
2U\langle\widehat\nabla U,X\rangle_{\widehat g},
\]
so the divergence theorem gives
\begin{align}
\int_{\widehat\Omega}
U^2\operatorname{div}_{\widehat g}X\,dV_{\widehat g}
&=
\int_{\widehat\Sigma}
U^2\langle X,\nu\rangle_{\widehat g}\,
dA_{\widehat\sigma}
-
2\int_{\widehat\Omega}
U\langle\widehat\nabla U,X\rangle_{\widehat g}\,
dV_{\widehat g}
\nonumber\\
&=
\int_{\widehat\Sigma}
\langle X,\nu\rangle_{\widehat g}\,
dA_{\widehat\sigma}
-
2\int_{\widehat\Omega}
U\langle\widehat\nabla U,X\rangle_{\widehat g}\,
dV_{\widehat g}.
\label{eq:divergence-U2X}
\end{align}
Substituting \eqref{eq:divergence-U2X} into
\eqref{eq:flux-step1} gives
\begin{align}
&8\int_{\widehat\Sigma}\partial_\nu U\,dA_{\widehat\sigma}
+
2\int_{\widehat\Sigma}
\langle X,\nu\rangle_{\widehat g}\,
dA_{\widehat\sigma}
\nonumber\\
&\qquad\geq
\int_{\widehat\Omega}
\left(
8|\widehat\nabla U|^2
+
2U^2|X|^2
+
4U\langle\widehat\nabla U,X\rangle
\right)dV_{\widehat g}.
\label{eq:flux-step2}
\end{align}
The integrand on the right satisfies
\begin{align}
8|\widehat\nabla U|^2
+
2U^2|X|^2
+
4U\langle\widehat\nabla U,X\rangle
&=
6|\widehat\nabla U|^2
+
2|\widehat\nabla U+UX|^2.
\end{align}
Therefore
\begin{equation}
\label{eq:flux-final}
4\int_{\widehat\Sigma}\partial_\nu U\,dA_{\widehat\sigma}
+
\int_{\widehat\Sigma}
\langle X,\nu\rangle_{\widehat g}\,
dA_{\widehat\sigma}
\geq
3\int_{\widehat\Omega}
|\widehat\nabla U|^2\,dV_{\widehat g}
+
\int_{\widehat\Omega}
|\widehat\nabla U+UX|^2\,dV_{\widehat g}.
\end{equation}
Now combining \eqref{eq:Jang-generalized-H} and
\eqref{eq:Hsf-boundary},
\begin{align}
H_{\rm sf}-\mathcal H
&=
\widehat H+4\partial_\nu U
-
\left(
\widehat H-\langle X,\nu\rangle_{\widehat g}
\right)
\nonumber\\
&=
4\partial_\nu U+\langle X,\nu\rangle_{\widehat g}.
\label{eq:Hsf-minus-Hgeneral}
\end{align}
Hence \eqref{eq:flux-final} gives
\begin{equation}
\label{eq:H-comparison-final}
\int_{\widehat\Sigma}
(H_{\rm sf}-\mathcal H)\,dA_{\widehat\sigma}
\geq
3\int_{\widehat\Omega}
|\widehat\nabla U|^2\,dV_{\widehat g}
+
\int_{\widehat\Omega}
|\widehat\nabla U+UX|^2\,dV_{\widehat g}
\geq0.
\end{equation}
Finally, from the reduced Wang--Yau identity
\[
8\pi E_{\rm WY}(\Sigma,\tau)
=
\int_{\widehat\Sigma}
(k_0-\mathcal H)\,dA_{\widehat\sigma},
\]
we add and subtract $H_{\rm sf}$:
\begin{align}
8\pi E_{\rm WY}(\Sigma,\tau)
&=
\int_{\widehat\Sigma}
(k_0-H_{\rm sf})\,dA_{\widehat\sigma}
+
\int_{\widehat\Sigma}
(H_{\rm sf}-\mathcal H)\,dA_{\widehat\sigma}
\nonumber\\
&\geq
8\pi
m_{\rm BY}
(\widehat\Sigma;\widehat\Omega,g_{\rm sf})
+
3\int_{\widehat\Omega}
|\widehat\nabla U|^2\,dV_{\widehat g}
\nonumber\\
&\qquad
+
\int_{\widehat\Omega}
|\widehat\nabla U+UX|^2\,dV_{\widehat g}.
\end{align}
Therefore
\begin{equation}
\label{eq:WY-BY-final-comparison}
E_{\rm WY}(\Sigma,\tau)
\geq
m_{\rm BY}
(\widehat\Sigma;\widehat\Omega,g_{\rm sf})
+
\frac{3}{8\pi}
\int_{\widehat\Omega}
|\widehat\nabla U|^2\,dV_{\widehat g}
+
\frac{1}{8\pi}
\int_{\widehat\Omega}
|\widehat\nabla U+UX|^2\,dV_{\widehat g}.
\end{equation}
In particular,
\[
E_{\rm WY}(\Sigma,\tau)
\geq
m_{\rm BY}
(\widehat\Sigma;\widehat\Omega,g_{\rm sf}).
\]
This completes the proof.
\end{proof}

\noindent The main point is that, in the previous proposition \ref{prop:WY-scalar-flat-reduction} (or equivalently the theorem (\ref{1})), we reduced the Wang-Yau mass to the Brown-york mass of a scalar-flat metric conformal to the Jang fill-in metric with a boundary-preserving conformal transformation. In the remaining part, we prove that this Brown-- York mass is non-negative.

\noindent In this section we compute the second variation of the Brown--York
mass at the Euclidean metric along scalar-flat deformations generated
by transverse-traceless tensors and preserving the induced boundary
metric. Before the second variation computation, we prove the second main theorem \ref{2}. Let
\[
\Omega_0\Subset\mathbb R^3,
\qquad
\Sigma_0:=\partial\Omega_0,
\]
be a smooth bounded domain with connected strictly convex boundary.
We denote by $\delta$ the Euclidean metric and set
\[
\gamma_0:=\delta|_{T\Sigma_0}.
\]
Let $\nu_0$ be the outward Euclidean unit normal and define
\[
A_0(Y,Z)
:=
\delta(\nabla^\delta_Y\nu_0,Z),
\qquad
H_0:=\operatorname{tr}_{\gamma_0}A_0.
\]
Thus
\[
A_0>0,
\qquad
H_0>0
\]
on $\Sigma_0$ by its convexity property. In the following two propositions, we identify the scalar flat class of metrics generated by $TT$ perturbations of the Euclidean data and subsequent conformal transformations. The first proposition deals with the TT-generated scalar-flat completion. Essentially propositions (\ref{prop:TT-scalar-flat-completion}) and (\ref{prop:physical-realization-TT}) accomplish the proof of the theorem \ref{2}.
\begin{proposition}
\label{prop:TT-scalar-flat-completion}
Let $(\Omega_0,\delta)\subset\mathbb R^3$ be a smooth bounded
Euclidean domain with boundary
\[
\Sigma_0:=\partial\Omega_0,
\qquad
\gamma_0:=\delta|_{T\Sigma_0}.
\]
Fix
\begin{equation}
\label{eq:TT-h-assumptions}
0\neq h\in
\mathcal{T}^{k,\alpha}(\Omega_0;S^2T^*\Omega_0),
\qquad
k\geq2,
\qquad
0<\alpha<1,
\end{equation}
Then there exists $\varepsilon>0$ and a unique smooth map
$
(-\varepsilon,\varepsilon)
\ni\lambda
\longmapsto
v_\lambda\in C^{k,\alpha}(\Omega_0)
$
sufficiently close to $1$, with $v_\lambda>0$ and
$v_\lambda|_{\Sigma_0}=1$, such that
\begin{equation}
\label{eq:vlambda-equation}
-8\Delta_{\delta+\lambda h}v_\lambda
+
R_{\delta+\lambda h}v_\lambda
=
0
\qquad\text{in }\Omega_0.
\end{equation}
Consequently,
$G(\lambda h)
:=
g_\lambda
:=
v_\lambda^4(\delta+\lambda h)$
defines a smooth one-parameter family of Riemannian metrics satisfying
\begin{equation}
\label{eq:TT-family-properties}
R_{g_\lambda}=0,
\qquad
g_\lambda|_{T\Sigma_0}=\gamma_0,
\qquad
g_0=\delta.
\end{equation}
Moreover,
\begin{equation}
\label{eq:v-lambda-quadratic}
v_\lambda
=
1+
O_{C^{k,\alpha}}(\lambda^2).
\end{equation}
\end{proposition}

\begin{proof}
For $|\lambda|$ sufficiently small, set
\begin{equation}
\label{eq:gbar-lambda}
\bar g_\lambda
:=
\delta+\lambda h.
\end{equation}
Since $\delta$ is positive definite and
\[
\|\lambda h\|_{C^{k,\alpha}(\Omega_0)}
\longrightarrow0
\qquad
\text{as }\lambda\longrightarrow0,
\]
there exists $\lambda_0>0$ such that $\bar g_\lambda$ is
Riemannian whenever $|\lambda|<\lambda_0$.
\noindent Let
\begin{equation}
\label{eq:Cka-zero}
C^{k,\alpha}_0(\Omega_0)
:=
\left\{
w\in C^{k,\alpha}(\Omega_0):
w|_{\Sigma_0}=0
\right\},
\end{equation}
and define
\begin{equation}
\label{eq:F-map-definition}
\mathcal F:
(-\lambda_0,\lambda_0)
\times
C^{k,\alpha}_0(\Omega_0)
\longrightarrow
C^{k-2,\alpha}(\Omega_0)
\end{equation}
by
\begin{equation}
\label{eq:F-lambda-w}
\mathcal F(\lambda,w)
:=
-8\Delta_{\delta+\lambda h}(1+w)
+
R_{\delta+\lambda h}(1+w).
\end{equation}
The dependence of the coefficients of
$\Delta_{\delta+\lambda h}$ and $R_{\delta+\lambda h}$ on
$\lambda$ is smooth in the indicated Hölder spaces. Hence
$\mathcal F$ is smooth in a neighborhood of $(0,0)$. Now since $R_\delta=0$,
\begin{equation}
\label{eq:F-zero}
\mathcal F(0,0)=0.
\end{equation}
For every
$\phi\in C^{k,\alpha}_0(\Omega_0)$,
\begin{equation}
\label{eq:DwF}
D_w\mathcal F(0,0)[\phi]
=
-8\Delta_\delta\phi.
\end{equation}
The Dirichlet operator
\begin{equation}
\label{eq:Dirichlet-isomorphism}
-8\Delta_\delta:
C^{k,\alpha}_0(\Omega_0)
\longrightarrow
C^{k-2,\alpha}(\Omega_0)
\end{equation}
is an isomorphism. Therefore the implicit function theorem gives
$\varepsilon>0$ and a unique smooth curve
\begin{equation}
\label{eq:w-lambda-curve}
(-\varepsilon,\varepsilon)
\ni\lambda
\longmapsto
w_\lambda\in C^{k,\alpha}_0(\Omega_0),
\qquad
w_0=0,
\end{equation}
such that
\begin{equation}
\label{eq:F-lambda-wlambda-zero}
\mathcal F(\lambda,w_\lambda)=0.
\end{equation}
Set $v_\lambda:=1+w_\lambda$.
After decreasing $\varepsilon$ if necessary,
\[
\|w_\lambda\|_{C^0(\Omega_0)}<\frac12,
\]
and hence
\[
v_\lambda>\frac12>0.
\]
Thus $v_\lambda$ is the unique positive solution sufficiently
close to $1$ of
\[
-8\Delta_{\delta+\lambda h}v_\lambda
+
R_{\delta+\lambda h}v_\lambda
=
0,
\qquad
v_\lambda|_{\Sigma_0}=1.
\]
Now define
\[
g_\lambda
=
v_\lambda^4(\delta+\lambda h).
\]
In dimension three the conformal scalar-curvature transformation law is
\begin{equation}
\label{eq:conformal-scalar-law}
R_{u^4g}
=
u^{-5}
\left(
-8\Delta_gu+R_gu
\right).
\end{equation}
Applying \eqref{eq:conformal-scalar-law} with
\[
g=\delta+\lambda h,
\qquad
u=v_\lambda,
\]
and using \eqref{eq:vlambda-equation}, we obtain
\begin{equation}
\label{eq:Rglambda-zero}
R_{g_\lambda}=0.
\end{equation}
We next verify that the induced boundary metric is independent of
$\lambda$. If $Y,Z\in T\Sigma_0$, then
\[
\begin{aligned}
g_\lambda(Y,Z)
&=
v_\lambda^4
\left(
\delta(Y,Z)+\lambda h(Y,Z)
\right)\\
&=
\delta(Y,Z),
\end{aligned}
\]
because
\[
v_\lambda|_{\Sigma_0}=1,
\qquad
h^T|_{\Sigma_0}=0.
\]
Therefore
\begin{equation}
\label{eq:boundary-metric-fixed}
g_\lambda|_{T\Sigma_0}
=
\gamma_0.
\end{equation}
At $\lambda=0$, uniqueness gives
$v_0\equiv1$, and hence
\[
g_0=\delta.
\]
It remains to determine the first-order behavior of $v_\lambda$.
Differentiating
\eqref{eq:vlambda-equation}
with respect to $\lambda$ at $\lambda=0$ gives
\begin{equation}
\label{eq:vdot-equation-pre}
-8\Delta_\delta\dot v_0
+
DR_\delta(h)
=
0,
\qquad
\dot v_0|_{\Sigma_0}=0.
\end{equation}
There is no contribution from the variation of the Laplacian acting
on $v_0$, since
\[
v_0\equiv1
\qquad\Longrightarrow\qquad
D\Delta_\delta(h)(1)=0.
\]
The linearization of scalar curvature at the Euclidean metric is
\begin{equation}
\label{eq:DR-delta-h-again}
DR_\delta(h)
=
-\Delta_\delta
\left(
\operatorname{tr}_\delta h
\right)
+
\partial^i\partial^jh_{ij}.
\end{equation}
By the TT conditions,
\[
\operatorname{tr}_\delta h=0,
\qquad
\partial^jh_{ij}=0,
\]
and therefore
\begin{equation}
\label{eq:DR-TT-zero}
DR_\delta(h)=0.
\end{equation}
Thus
\begin{equation}
\label{eq:vdot-harmonic}
\Delta_\delta\dot v_0=0,
\qquad
\dot v_0|_{\Sigma_0}=0.
\end{equation}
Uniqueness for the Dirichlet Laplacian yields
\[
\dot v_0=0.
\]
Since
$\lambda\longmapsto v_\lambda$
is smooth as a map into $C^{k,\alpha}(\Omega_0)$, fundamental theorem of calculus or Taylor's theorem
gives
\[
v_\lambda
=
1+
O_{C^{k,\alpha}}(\lambda^2).
\]
More precisely, after decreasing $\varepsilon$ if necessary, there is
a constant $C_h>0$ such that
\begin{equation}
\label{eq:v-lambda-quadratic-estimate}
\|v_\lambda-1\|_{C^{k,\alpha}(\Omega_0)}
\leq
C_h\lambda^2
\end{equation}
for $|\lambda|<\varepsilon$.
Finally,
\[
v_\lambda^4
=
1+
O_{C^{k,\alpha}}(\lambda^2),
\]
and hence
\[
\begin{aligned}
g_\lambda
&=
v_\lambda^4(\delta+\lambda h)\\
&=
\delta+\lambda h
+
O_{C^{k,\alpha}}(\lambda^2).
\end{aligned}
\]
This completes the proof.
\end{proof}
\noindent Now we identify the scalar-flat Jang metric with the TT-generated family in the previous proposition \ref{}. The main point is that the Wang-Yau definition of the quasi-local mass involves a particular minimizing choice of the function $\tau$ that is essentially obtained by solving a fourth-order equation. First, recall the following definition 
\[
{\mathcal T}^{k,\alpha}
:=
\left\{
h\in C^{k,\alpha}
(\Omega_0;S^2T^*\Omega_0): ||h||_{C^{k,\alpha}}=1,
\operatorname{div}_{\delta}h=0,\;
\operatorname{tr}_{\delta}h=0,\;
h^T|_{\Sigma_0}=0
\right\},
\]
\begin{definition}[TT-generated Jang-reduced data]
\label{physical}
Let $(M,g,K)$ be a smooth physical initial data set satisfying the
Einstein constraint equations and the dominant energy condition, and let
$\Sigma=\partial\Omega$. 
Let
$(\widehat\Omega_*,\widehat g_*)$
be the Jang fill-in associated with an admissible $\tau$, and let $U_*>0$ be the
solution of
\[
-8\Delta_{\widehat g_*}U_*
+R_{\widehat g_*}U_*=0
\quad\hbox{in }\widehat\Omega_*,
\qquad
U_*=1
\quad\hbox{on }\partial\widehat\Omega_*.
\]
Set $g_{{\rm sf},*}:=U_*^4\widehat g_*$.
Let $(\Omega_0,\delta)\subset{\mathbb R}^3$ be the strictly convex
Euclidean fill-in determined by the projected reference embedding
corresponding to $\tau$.  We say that the pair
$\bigl((M,g,K),\tau\bigr)$
belongs to the TT-generated Jang-reduced class if there exists
a diffeomorphism
\[
\Phi:\Omega_0\longrightarrow\widehat\Omega_*,
\qquad
\Phi|_{\Sigma_0}=\mathrm{Id},
\]
and a tensor $0\neq h\in {\mathcal T}^{k,\alpha}$, $\mathbb{R}^{+}\ni|\lambda|\ll 1$,~$k\geq 2,\alpha\in(0,1)$
such that
\[
\Phi^*g_{{\rm sf},*}=v_{\lambda,h}^4(\delta+\lambda h).
\]
where $v_{\lambda,h}>0$ is the unique solution close to $1$ of
\[
-8\Delta_{\delta+\lambda h}v_{\lambda,h}+R_{\delta+\lambda h}v_{\lambda,h}=0,
\qquad
v_{\lambda,h}|_{\Sigma_0}=1.
\]
\end{definition}

\begin{proposition}[Physical realization of the TT-generated class]
\label{prop:physical-realization-TT}
Let $(\Omega_0,\delta)\subset\mathbb R^3$ be a smooth bounded domain
with strictly convex boundary
$\Sigma_0:=\partial\Omega_0,
\gamma_0:=\delta|_{T\Sigma_0}$,
and let
$0\neq h\in \mathcal{T}^{k,\alpha}
(\Omega_0;S^2T^*\Omega_0),
k\geq2,
0<\alpha<1$.
Let
\[
g_{{\rm sf},\lambda}
:=
G(\lambda h)
=
v_\lambda^4(\delta+\lambda h)
\]
be the scalar-flat family constructed in
Proposition~\ref{prop:TT-scalar-flat-completion}. Then, after decreasing $\varepsilon>0$ if necessary, there exists
$a_0>0$ such that for every
$0<|\lambda|<\varepsilon,
\frac{a_0}{2}<a<\frac{3a_0}{2}$,
there is an open neighborhood of $(0,0)$ in the variables $(f,S)$
for which one obtains smooth physical initial data $(\Omega_0,g_{\lambda,a,f},K_{\lambda,a,f,S})$
satisfying the Einstein constraint equations and the strict dominant
energy condition. Moreover, the function $f$ is an exact solution of
Jang's equation for these data, and the subsequent conformal
scalar-flat reduction is precisely $g_{{\rm sf},\lambda}$.
\end{proposition}
\begin{remark}
The main point of the proposition is to explicitly prove the existence of physical data in the $ TT$- generated class that verify the dominant energy condition as stated in the definition \ref{physical}. We can make it more explicit here. More explicitly, let $\phi_\lambda$ be the solution of
\begin{equation}
\label{eq:phi-inverse-construction}
\left\{
\begin{aligned}
-\Delta_{g_{{\rm sf},\lambda}}\phi_\lambda
&=1
&&\text{in }\Omega_0,\\
\phi_\lambda&=0
&&\text{on }\Sigma_0,
\end{aligned}
\right.
\end{equation}
and set
$w_{\lambda,a}
:=
1+a\phi_\lambda,
\widehat g_{\lambda,a}
:=
w_{\lambda,a}^4g_{{\rm sf},\lambda}$.
For $f\in C^{k+1,\alpha}(\Omega_0)$ sufficiently small $df$ in $C^{k,\alpha}$, define
\begin{equation}
\label{eq:physical-g-inverse-construction}
g_{\lambda,a,f}
:=
\widehat g_{\lambda,a}
-
df\otimes df.
\end{equation}
Let $A_f$ denote the second fundamental form of the graph of $f$ in
$(\Omega_0\times\mathbb R,
g_{\lambda,a,f}+dt^2)$,
and let
$S\in C^{k-1,\alpha}
(\Omega_0;S^2T^*\Omega_0)$
be sufficiently small and satisfy
\begin{equation}
\label{eq:S-hat-tracefree}
\operatorname{tr}_{\widehat g_{\lambda,a}}S=0.
\end{equation}
Define
\begin{equation}
\label{eq:physical-K-inverse-construction}
K_{\lambda,a,f,S}
:=
A_f+S.
\end{equation}
Then $f$ satisfies Jang's equation with boundary value
\[
\tau:=f|_{\Sigma_0},
\]
the induced metric on its Jang graph is exactly
$\widehat g_{\lambda,a}$, and
\begin{equation}
\label{eq:exact-return-TT}
U_{\lambda,a}^4\widehat g_{\lambda,a}
=
g_{{\rm sf},\lambda},
\qquad
U_{\lambda,a}:=w_{\lambda,a}^{-1},
\end{equation}
where $U_{\lambda,a}$ is the unique positive solution of
\begin{equation}
\label{eq:U-inverse-construction}
-8\Delta_{\widehat g_{\lambda,a}}U_{\lambda,a}
+
R_{\widehat g_{\lambda,a}}U_{\lambda,a}
=
0,
\qquad
U_{\lambda,a}|_{\Sigma_0}=1.
\end{equation}
Consequently, these initial data belong to the TT-generated
Jang-reduced class of Definition~\ref{physical}.
\end{remark}

\begin{proof}
Since $g_{{\rm sf},\lambda}$ is scalar flat, the maximum principle
applied to \eqref{eq:phi-inverse-construction} gives
\[
\phi_\lambda>0
\qquad\text{in }\Omega_0.
\]
For $a>0$, the conformal transformation law yields since $R_{g_{sf,\lambda}}=0$
\begin{align}
R_{\widehat g_{\lambda,a}}
&=
w_{\lambda,a}^{-5}
\left(
-8\Delta_{g_{{\rm sf},\lambda}}w_{\lambda,a}
+
R_{g_{{\rm sf},\lambda}}w_{\lambda,a}
\right)
\nonumber=
8a\,w_{\lambda,a}^{-5}
>0.
\label{eq:positive-R-hat-construction}
\end{align}
Since $\phi_\lambda=0$ on $\Sigma_0$,
\[
w_{\lambda,a}|_{\Sigma_0}=1,
\qquad
\widehat g_{\lambda,a}|_{T\Sigma_0}
=
\gamma_0.
\]
For $df$ sufficiently small in $C^{k,\alpha}$ one has
\[
|df|_{\widehat g_{\lambda,a}}<1,
\]
and therefore
\[
g_{\lambda,a,f}
=
\widehat g_{\lambda,a}-df\otimes df
\]
is positive definite. By construction,
\begin{equation}
\label{eq:graph-metric-exact}
g_{\lambda,a,f}
+
df\otimes df
=
\widehat g_{\lambda,a}.
\end{equation}
Hence the metric induced on the graph of $f$ in
$(\Omega_0\times\mathbb R,g_{\lambda,a,f}+dt^2)$ is precisely
$\widehat g_{\lambda,a}$. Now let $A_f$ be the second fundamental form of this graph. From
\eqref{eq:physical-K-inverse-construction} and
\eqref{eq:S-hat-tracefree},
\[
\operatorname{tr}_{\widehat g_{\lambda,a}}
(A_f-K_{\lambda,a,f,S})
=
-\operatorname{tr}_{\widehat g_{\lambda,a}}S
=
0.
\]
Thus $f$ solves Jang's equation exactly. Next, we set
\[
U_{\lambda,a}=w_{\lambda,a}^{-1}.
\]
Since $w_{\lambda,a}=1$ on $\Sigma_0$,
\[
U_{\lambda,a}|_{\Sigma_0}=1,
\]
while
\[
U_{\lambda,a}^4\widehat g_{\lambda,a}
=
w_{\lambda,a}^{-4}
w_{\lambda,a}^4
g_{{\rm sf},\lambda}
=
g_{{\rm sf},\lambda}.
\]
Because the metric on the right is scalar flat, conformal covariance
of the scalar-curvature operator gives
\[
-8\Delta_{\widehat g_{\lambda,a}}U_{\lambda,a}
+
R_{\widehat g_{\lambda,a}}U_{\lambda,a}
=
0.
\]
Hence the Jang--conformal reduction of the constructed physical data
returns exactly
$G(\lambda h)$. It remains to verify the constraints and the dominant energy
condition. Define
\begin{align}
2\mu_{\lambda,a,f,S}
&:=
R_{g_{\lambda,a,f}}
-
|K_{\lambda,a,f,S}|_{g_{\lambda,a,f}}^2
+
\left(
\operatorname{tr}_{g_{\lambda,a,f}}
K_{\lambda,a,f,S}
\right)^2,
\label{eq:mu-constructed}
\\
J_{\lambda,a,f,S}
&:=
\operatorname{div}_{g_{\lambda,a,f}}
K_{\lambda,a,f,S}
-
d\left(
\operatorname{tr}_{g_{\lambda,a,f}}
K_{\lambda,a,f,S}
\right).
\label{eq:J-constructed}
\end{align}
The Einstein constraint equations then hold identically. At
\[
f=0,
\qquad
S=0,
\]
one has
\[
g_{\lambda,a,0}
=
\widehat g_{\lambda,a},
\qquad
K_{\lambda,a,0,0}=0,
\]
and hence, we have from (\ref{eq:mu-constructed})-(\ref{eq:J-constructed})
\begin{equation}
\label{eq:strict-DEC-base}
\mu_{\lambda,a,0,0}
=
\frac12R_{\widehat g_{\lambda,a}}
=
4a\,w_{\lambda,a}^{-5}
>0,
\qquad
J_{\lambda,a,0,0}=0.
\end{equation}
After fixing $a_0>0$ and decreasing $\varepsilon$ and $a_0$ if
necessary, the quantity on the right-hand side of
\eqref{eq:strict-DEC-base} has a uniform positive lower bound for
\[
|\lambda|<\varepsilon,
\qquad
\frac{a_0}{2}<a<\frac{3a_0}{2}.
\]
The constraint map
\[
(g,K)
\longmapsto
\left(
\frac12
\left[
R_g-|K|_g^2+(\operatorname{tr}_gK)^2
\right],
\,
\operatorname{div}_gK-d(\operatorname{tr}_gK)
\right)
\]
is continuous as a map $
C^{k,\alpha}\times C^{k-1,\alpha}\to C^{k-2,\alpha}\times C^{k-2,\alpha}$.
Therefore, for $df$ and $S$ sufficiently small in their respective strong enough topology,
\begin{equation}
\label{eq:strict-DEC-open}
\mu_{\lambda,a,f,S}
>
|J_{\lambda,a,f,S}|_{g_{\lambda,a,f}}
\qquad
\text{on }\Omega_0.
\end{equation}
Thus the constructed data satisfy the strict dominant energy
condition. Finally, if $\tau=f|_{\Sigma_0}$ and
\[
\sigma
:=
g_{\lambda,a,f}|_{T\Sigma_0},
\]
then
\begin{equation}
\label{eq:projected-boundary-exact}
\sigma+d\tau\otimes d\tau
=
\widehat g_{\lambda,a}|_{T\Sigma_0}
=
\gamma_0.
\end{equation}
Hence the projected Wang--Yau boundary metric is exactly the fixed
strictly convex Euclidean metric. At $(f,S)=(0,0)$ the generalized
boundary mean curvature is positive for $a_0$ sufficiently small;
this remains true after shrinking the above neighborhood, again by
continuity. Thus $\tau$ is admissible. The admissible parameter set for sufficiently small $\eta$
\[
\left\{
(\lambda,a,f,S):
0<|\lambda|<\varepsilon,\;
\frac{a_0}{2}<a<\frac{3a_0}{2},\;
\|df\|_{C^{k,\alpha}}<\eta,\;
\|S\|_{C^{k-1,\alpha}}<\eta,\;
\operatorname{tr}_{\widehat g_{\lambda,a}}S=0
\right\}
\]
is open in the corresponding parameter space, and every member gives
genuine constraint-satisfying initial data with strict dominant
energy condition whose Jang-conformal reduction is
$G(\lambda h)$. This completes the proof and the construction of the physical data.
\end{proof}

\noindent Now we move on with the computation of the variations of the Brown-York mass for the conformal Jang data. Since the boundary metric is independent of $\lambda$, its Euclidean
isometric embedding is independent of $\lambda$, up to Euclidean
rigid motions.  Hence the reference mean curvature $H_0$ and the area
form $dA_{\gamma_0}$ are fixed.  If $H_\lambda$ denotes the mean
curvature of
\[
\Sigma_0\subset(\Omega_0,g_\lambda)
\]
with respect to the outward $g_\lambda$--unit normal, then
\begin{equation}
\label{eq:BY-glambda}
m_{\rm BY}(g_\lambda)
=
\frac{1}{8\pi}
\int_{\Sigma_0}
\left(
H_0-H_\lambda
\right)dA_{\gamma_0}.
\end{equation}
The computation of the Hessian of
$m_{\rm BY}(g_\lambda)$ at $\lambda=0$ is carried out below. For later use, define the tangential vector field
$Z_h\in\Gamma(T\Sigma_0)$ by
\begin{equation}
\label{eq:Z-h}
        \gamma_0(Z_h,Y)
        =
        h(\nu_0,Y),
        \qquad
        Y\in T\Sigma_0.
\end{equation}
In the next lemma, we obtain the first variation formula for the mean curvature under the scalar flat constraint.
\begin{lemma}
\label{lem:BY-first-variation-scalar-flat}

Let $I\subset\mathbb R$ be an open interval containing $0$, and let
\[
\lambda\longmapsto g_\lambda
\]
be a $C^2$ family of Riemannian metrics on $\Omega_0$ such that, for every
$\lambda\in I$,
\[
R_{g_\lambda}=0,
\qquad
g_\lambda|_{T\Sigma_0}=\gamma_0.
\]
Let
\[
q_\lambda:=\partial_\lambda g_\lambda,
\]
and let $H_\lambda$ denote the mean curvature of
$\Sigma_0\subset(\Omega_0,g_\lambda)$ with respect to the outward
$g_\lambda$--unit normal. Then
\begin{equation}
\label{eq:total-H-first-variation}
2\frac{d}{d\lambda}
\int_{\Sigma_0}H_\lambda\,dA_{\gamma_0}
=
-
\int_{\Omega_0}
\left\langle
q_\lambda,\text{Ric}(g_\lambda)
\right\rangle_{g_\lambda}
\,dV_{g_\lambda}.
\end{equation}
\end{lemma}
\begin{proof}
Fix $\lambda\in I$ and, for the duration of the computation, suppress
the subscript $\lambda$.  Thus
\[
g=g_\lambda,
\qquad
q=\partial_\lambda g_\lambda,
\qquad
\nu=\nu_\lambda,
\qquad
A=A_\lambda,
\qquad
H=H_\lambda,
\]
and all covariant derivatives, traces, divergences, and contractions
below are taken with respect to $g$ unless explicitly indicated
otherwise. Since
\[
g_\lambda|_{T\Sigma_0}=\gamma_0
\]
is independent of $\lambda$, differentiation gives
\begin{equation}
\label{eq:qT-zero-first-var}
q^T=0
\qquad\text{on }\Sigma_0.
\end{equation}
In particular, the induced area form is independent of $\lambda$.
Indeed, if $\gamma_\lambda=g_\lambda|_{T\Sigma_0}$, then
\[
\partial_\lambda dA_{\gamma_\lambda}
=
\frac12
\operatorname{tr}_{\gamma_\lambda}
(\partial_\lambda\gamma_\lambda)
\,dA_{\gamma_\lambda}
=
\frac12\operatorname{tr}_{\gamma_0}(q^T)\,dA_{\gamma_0}
=0.
\]
Consequently,
\begin{equation}
\label{eq:H-integral-derivative}
\frac{d}{d\lambda}
\int_{\Sigma_0}H_\lambda\,dA_{\gamma_0}
=
\int_{\Sigma_0}\dot H\,dA_{\gamma_0}.
\end{equation}
We first compute $\dot H$. We provide the explicit computations even if some of them are standard for self consistency. Define the tangential vector field
$X_q\in T\Sigma_0$ by
\begin{equation}
\label{eq:Xq-definition}
\gamma_0(X_q,Y)=q(\nu,Y),
\qquad
Y\in T\Sigma_0.
\end{equation}
Differentiating
\[
g(\nu,\nu)=1
\]
gives
\[
q(\nu,\nu)+2g(\dot\nu,\nu)=0,
\]
hence
\begin{equation}
\label{eq:nudot-normal}
g(\dot\nu,\nu)
=
-\frac12q(\nu,\nu).
\end{equation}
Likewise, for every fixed tangential vector field
$Y\in T\Sigma_0$,
\[
g(\nu,Y)=0,
\]
and therefore
\[
q(\nu,Y)+g(\dot\nu,Y)=0.
\]
By \eqref{eq:Xq-definition},
\[
(\dot\nu)^T=-X_q.
\]
Combining this with \eqref{eq:nudot-normal},
\begin{equation}
\label{eq:nudot}
\dot\nu
=
-X_q-\frac12q(\nu,\nu)\nu.
\end{equation}
Let $\{e_1,e_2\}$ be a local tangent frame which is
$g$--orthonormal at the point under consideration.  We may further
choose it so that
\[
\nabla^{\Sigma}_{e_a}e_b=0
\]
at that point.  The second fundamental form is
\[
A_{ab}
=
g(\nabla_{e_a}\nu,e_b).
\]
The variation of the Levi--Civita connection is determined by
\begin{equation}
\label{eq:connection-variation}
2g\bigl((\partial_\lambda\nabla)_YZ,W\bigr)
=
(\nabla_Yq)(Z,W)
+
(\nabla_Zq)(Y,W)
-
(\nabla_Wq)(Y,Z).
\end{equation}
For completeness, this follows by differentiating the Koszul formula.
Applying \eqref{eq:connection-variation} with
\[
Y=e_a,\qquad Z=\nu,\qquad W=e_b
\]
gives
\begin{equation}
\label{eq:connection-var-boundary}
g\bigl((\partial_\lambda\nabla)_{e_a}\nu,e_b\bigr)
=
\frac12
\left[
(\nabla_{e_a}q)(\nu,e_b)
+
(\nabla_\nu q)(e_a,e_b)
-
(\nabla_{e_b}q)(e_a,\nu)
\right].
\end{equation}
Differentiating $A_{ab}$, we obtain
\begin{align}
\dot A_{ab}
&=
q(\nabla_{e_a}\nu,e_b)
+
g\bigl((\partial_\lambda\nabla)_{e_a}\nu,e_b\bigr)
+
g(\nabla_{e_a}\dot\nu,e_b).
\label{eq:Adot-start}
\end{align}
Since
\[
\nabla_{e_a}\nu=A_a{}^c e_c
\]
and $q^T=0$ on $\Sigma_0$,
\[
q(\nabla_{e_a}\nu,e_b)
=
A_a{}^c q(e_c,e_b)
=0.
\]
Furthermore, using \eqref{eq:nudot},
\begin{align}
g(\nabla_{e_a}\dot\nu,e_b)
&=
-g(\nabla_{e_a}X_q,e_b)
-\frac12
g\bigl(
\nabla_{e_a}(q(\nu,\nu)\nu),e_b
\bigr)
\nonumber\\
&=
-g(\nabla_{e_a}X_q,e_b)
-\frac12q(\nu,\nu)
g(\nabla_{e_a}\nu,e_b)
\nonumber\\
&=
-g(\nabla_{e_a}X_q,e_b)
-\frac12q(\nu,\nu)A_{ab}.
\label{eq:nudot-contribution}
\end{align}
Substitution of
\eqref{eq:connection-var-boundary} and
\eqref{eq:nudot-contribution} into \eqref{eq:Adot-start} yields
\begin{align}
\dot A_{ab}
={}&
\frac12
\left[
(\nabla_{e_a}q)(\nu,e_b)
+
(\nabla_\nu q)(e_a,e_b)
-
(\nabla_{e_b}q)(e_a,\nu)
\right]
\nonumber\\
&-
g(\nabla_{e_a}X_q,e_b)
-\frac12q(\nu,\nu)A_{ab}.
\label{eq:Adot-full}
\end{align}
Since the induced metric $\gamma_\lambda$ is independent of
$\lambda$,
\[
\partial_\lambda\gamma_\lambda^{ab}=0.
\]
Thus
\[
\dot H
=
\gamma_0^{ab}\dot A_{ab}.
\]
Tracing \eqref{eq:Adot-full}, the first and third terms cancel
$
\sum_{a=1}^{2}
(\nabla_{e_a}q)(\nu,e_a)
=
\sum_{a=1}^{2}
(\nabla_{e_a}q)(e_a,\nu),
$
because $q$ is symmetric.  Therefore
\begin{equation}
\label{eq:Hdot-first-form}
2\dot H
=
\sum_{a=1}^{2}
(\nabla_\nu q)(e_a,e_a)
-
2\operatorname{div}_{\Sigma_0}X_q
-
Hq(\nu,\nu).
\end{equation}
We now rewrite the first and third terms in invariant form.  By
definition,
\begin{align}
d(\operatorname{tr}_gq)(\nu)
&=
(\nabla_\nu q)(\nu,\nu)
+
\sum_{a=1}^{2}
(\nabla_\nu q)(e_a,e_a).
\label{eq:dtrq-normal}
\end{align}
Moreover,
\begin{align}
(\operatorname{div}_gq)(\nu)
&=
\sum_{i=1}^{3}
(\nabla_{e_i}q)(e_i,\nu)
\nonumber\\
&=
(\nabla_\nu q)(\nu,\nu)
+
\sum_{a=1}^{2}
(\nabla_{e_a}q)(e_a,\nu).
\label{eq:divq-normal}
\end{align}
It remains to compute the last sum.  At the chosen point,
\[
\nabla^\Sigma_{e_a}e_a=0.
\]
With our sign convention
\[
A(Y,Z)=g(\nabla_Y\nu,Z),
\]
and
\[
\nabla_{e_a}e_b
=
\nabla^\Sigma_{e_a}e_b-A_{ab}\nu.
\]
Hence
\[
\nabla_{e_a}e_a=-A_{aa}\nu
\]
at the chosen point, while
\[
\nabla_{e_a}\nu=A_a{}^b e_b.
\]
Consequently,
\begin{align}
(\nabla_{e_a}q)(e_a,\nu)
&=
e_a\!\left(q(e_a,\nu)\right)
-
q(\nabla_{e_a}e_a,\nu)
-
q(e_a,\nabla_{e_a}\nu)
\nonumber\\
&=
e_a\!\left(q(e_a,\nu)\right)
+
A_{aa}q(\nu,\nu)
-
A_a{}^bq(e_a,e_b).
\end{align}
Since $q^T=0$,
\[
q(e_a,e_b)=0,
\]
and thus
\[
(\nabla_{e_a}q)(e_a,\nu)
=
e_a\!\left(q(e_a,\nu)\right)
+
A_{aa}q(\nu,\nu).
\]
Summing over $a$ gives
\begin{equation}
\label{eq:divq-boundary-decomp}
\sum_{a=1}^{2}
(\nabla_{e_a}q)(e_a,\nu)
=
\operatorname{div}_{\Sigma_0}X_q
+
Hq(\nu,\nu).
\end{equation}
Subtracting \eqref{eq:divq-normal} from
\eqref{eq:dtrq-normal} and using
\eqref{eq:divq-boundary-decomp}, we obtain
\begin{align}
&
\left(
d\operatorname{tr}_gq-\operatorname{div}_gq
\right)(\nu)
\nonumber\\
&\qquad=
\sum_{a=1}^{2}
(\nabla_\nu q)(e_a,e_a)
-
\operatorname{div}_{\Sigma_0}X_q
-
Hq(\nu,\nu).
\label{eq:dtr-div-boundary}
\end{align}
Comparing \eqref{eq:Hdot-first-form} with
\eqref{eq:dtr-div-boundary} yields the pointwise boundary variation
formula
\begin{equation}
\label{eq:Hdot-final}
2\dot H
=
\left(
d\operatorname{tr}_gq-\operatorname{div}_gq
\right)(\nu)
-
\operatorname{div}_{\Sigma_0}X_q.
\end{equation}
Since $\Sigma_0$ is closed,
\[
\int_{\Sigma_0}
\operatorname{div}_{\Sigma_0}X_q\,dA_{\gamma_0}
=0.
\]
Hence, by \eqref{eq:H-integral-derivative},
\begin{equation}
\label{eq:Hvariation-integrated}
2\frac{d}{d\lambda}
\int_{\Sigma_0}H_\lambda\,dA_{\gamma_0}
=
\int_{\Sigma_0}
\left(
d\operatorname{tr}_gq-\operatorname{div}_gq
\right)(\nu)\,
dA_{\gamma_0}.
\end{equation}
We now use scalar flatness.  Since
\[
R_{g_\lambda}=0
\qquad\text{for every }\lambda,
\]
we have pointwise
\[
\partial_\lambda R_{g_\lambda}=0.
\]
Standard computation yields
\begin{equation}
\label{eq:scalar-curvature-linearization-full}
DR_g(q)
=
-\Delta_g(\operatorname{tr}_gq)
+
\operatorname{div}_g\operatorname{div}_gq
-
\langle q,\operatorname{Ric}(g)\rangle_g.
\end{equation}
Because $R_{g_\lambda}\equiv0$,
\[
DR_g(q)=0,
\]
and therefore
\begin{equation}
\label{eq:DR-zero}
0
=
-\Delta_g(\operatorname{tr}_gq)
+
\operatorname{div}_g\operatorname{div}_gq
-
\langle q,\operatorname{Ric}(g)\rangle_g.
\end{equation}
Integrating \eqref{eq:DR-zero} over $\Omega_0$ gives
\begin{align}
0
={}&
-\int_{\Omega_0}
\Delta_g(\operatorname{tr}_gq)\,dV_g
+
\int_{\Omega_0}
\operatorname{div}_g\operatorname{div}_gq\,dV_g
\nonumber\\
&-
\int_{\Omega_0}
\langle q,\operatorname{Ric}(g)\rangle_g\,dV_g.
\end{align}
By the divergence theorem,
\[
\int_{\Omega_0}
\Delta_g(\operatorname{tr}_gq)\,dV_g
=
\int_{\Sigma_0}
d(\operatorname{tr}_gq)(\nu)\,dA_{\gamma_0},
\]
and, regarding $\operatorname{div}_gq$ as a one-form,
\[
\int_{\Omega_0}
\operatorname{div}_g\operatorname{div}_gq\,dV_g
=
\int_{\Sigma_0}
(\operatorname{div}_gq)(\nu)\,dA_{\gamma_0}.
\]
Hence
\begin{align}
0
={}&
-\int_{\Sigma_0}
d(\operatorname{tr}_gq)(\nu)\,dA_{\gamma_0}
+
\int_{\Sigma_0}
(\operatorname{div}_gq)(\nu)\,dA_{\gamma_0}
\nonumber\\
&-
\int_{\Omega_0}
\langle q,\operatorname{Ric}(g)\rangle_g\,dV_g,
\end{align}
or equivalently,
\begin{equation}
\label{eq:boundary-bulk-identity}
\int_{\Sigma_0}
\left(
d\operatorname{tr}_gq-\operatorname{div}_gq
\right)(\nu)\,dA_{\gamma_0}
=
-
\int_{\Omega_0}
\langle q,\operatorname{Ric}(g)\rangle_g\,dV_g.
\end{equation}
Combining \eqref{eq:Hvariation-integrated} with
\eqref{eq:boundary-bulk-identity} yields
\[
2\frac{d}{d\lambda}
\int_{\Sigma_0}H_\lambda\,dA_{\gamma_0}
=
-
\int_{\Omega_0}
\left\langle
q_\lambda,\operatorname{Ric}(g_\lambda)
\right\rangle_{g_\lambda}
\,dV_{g_\lambda},
\]
which is \eqref{eq:total-H-first-variation}.
\end{proof}

\begin{remark}
For any smooth one-parameter family $g(t)$ satisfying
$R_{g(t)}=0, g(t)|_{T\partial\Omega}=\gamma$
with fixed boundary metric, Lemma~\ref{lem:BY-first-variation-scalar-flat} gives
\[
2\frac{d}{dt}\int_{\partial\Omega}H_{g(t)}\,dA_{\gamma}
=
-\int_{\Omega}
\langle \dot g,\operatorname{Ric}(g)\rangle_g\,dV_g.
\]
Thus, if such a family were to satisfy the Ricci-flow equation
$\dot g=-2\operatorname{Ric}(g)$, then
\[
\frac{d}{dt}m_{\mathrm{BY}}(g(t))
=
-\frac{1}{8\pi}
\int_{\Omega}
|\operatorname{Ric}(g(t))|^2\,dV_{g(t)}
\leq 0.
\]
However, an ordinary Ricci flow cannot remain in the scalar-flat class
unless it is Ricci flat, since
\[
\partial_t R=\Delta R+2|\operatorname{Ric}|^2.
\]
Hence any genuinely nontrivial evolution preserving both scalar
flatness and the induced boundary metric would necessarily require a
constrained modification of Ricci flow such as Fischer's conformal Ricci flow for the scalar flat case. The existence and well-posedness
of such a boundary-value problem on a compact manifold with boundary
are delicate and are not addressed here.
\end{remark}

\begin{remark}
A natural \textbf{formal} evolution within the scalar-flat class is the
projected Ricci flow. Let
\[
\mathcal S_{\gamma_0}
:=
\left\{
g:
R_g=0,
\qquad
g|_{T\partial\Omega}=\gamma_0
\right\},
\]
and let $\Pi_g$ denote the $L^2$-orthogonal projection onto the formal
tangent space
\[
T_g\mathcal S_{\gamma_0}
=
\left\{
q:
DR_g(q)=0,
\qquad
q^T|_{\partial\Omega}=0
\right\}.
\]
Formally, one may then consider
\[
\partial_t g
=
-2\Pi_g\operatorname{Ric}(g).
\]
By construction, this evolution preserves both scalar flatness and the
induced boundary metric. Moreover, Lemma~\ref{lem:BY-first-variation-scalar-flat}
gives
\[
\frac{d}{dt}m_{\mathrm{BY}}(g(t))
=
\frac{1}{16\pi}
\int_{\Omega}
\left\langle
\partial_t g,
\operatorname{Ric}(g)
\right\rangle_g
\,dV_g,
\]
and therefore
\[
\frac{d}{dt}m_{\mathrm{BY}}(g(t))
=
-\frac{1}{8\pi}
\left\|
\Pi_g\operatorname{Ric}(g)
\right\|_{L^2(\Omega,g)}^2
\leq 0.
\]
Thus the Brown--York mass is formally monotone nonincreasing along the
projected Ricci flow. The main analytic subtlety is whether the
projection $\Pi_g$ can be realized so that the resulting evolution is
a well-posed parabolic boundary-value problem on a compact manifold
with boundary. We do not address this issue here.
\end{remark}

\noindent We can now compute the Brown--York Hessian along the transverse-traceless direction. 
\begin{lemma}
\label{lem:TT-BY-Hessian}

Let $g_\lambda$ be the scalar-flat family, generated by
$0\neq h\in \mathcal{T}^{k,\alpha}$.  Define
$Z_h\in\Gamma(T\Sigma_0)$ by
\[
\gamma_0(Z_h,Y)=h(\nu_0,Y),
\qquad Y\in T\Sigma_0.
\]
Then
\[
\left.
\frac{d}{d\lambda}
\right|_{\lambda=0}
m_{\rm BY}(g_\lambda)
=0,
\]
and
\begin{equation}
\label{eq:BY-TT-second}
8\pi
\left.
\frac{d^2}{d\lambda^2}
\right|_{\lambda=0}
m_{\rm BY}(g_\lambda)
=
\frac14
\int_{\Omega_0}
|\nabla^\delta h|_\delta^2\,dV_\delta
+
\frac12
\int_{\Sigma_0}
\left(
A_0(Z_h,Z_h)
+
H_0|Z_h|_{\gamma_0}^2
\right)dA_{\gamma_0}.
\end{equation}
In particular, since $A_0>0$ and $H_0>0$ on $\Sigma_0$,
\[
\left.
\frac{d^2}{d\lambda^2}
\right|_{\lambda=0}
m_{\rm BY}(g_\lambda)
>0.
\]

\end{lemma}

\begin{proof}
Set
\begin{equation}
\label{eq:F-total-mean}
F(\lambda)
:=
\int_{\Sigma_0}H_\lambda\,dA_{\gamma_0}.
\end{equation}
Since the induced boundary metric is independent of $\lambda$,
\[
m_{\rm BY}(g_\lambda)
=
\frac{1}{8\pi}
\int_{\Sigma_0}(H_0-H_\lambda)\,dA_{\gamma_0},
\]
and therefore
\begin{equation}
\label{eq:BY-vs-F}
8\pi m_{\rm BY}'(\lambda)=-F'(\lambda),
\qquad
8\pi m_{\rm BY}''(\lambda)=-F''(\lambda).
\end{equation}
By Lemma~\ref{lem:BY-first-variation-scalar-flat},
\begin{equation}
\label{eq:Fprime-Ric}
2F'(\lambda)
=
-
\int_{\Omega_0}
\left\langle
\dot g_\lambda,\text{Ric}(g_\lambda)
\right\rangle_{g_\lambda}
\,dV_{g_\lambda}.
\end{equation}
At $\lambda=0$,
\[
g_0=\delta,
\qquad
\dot g_0=h,
\qquad
\text{Ric}(\delta)=0,
\]
hence
\begin{equation}
\label{eq:Fprime-zero}
F'(0)=0,
\qquad
m_{\rm BY}'(0)=0.
\end{equation}

\noindent Differentiating \eqref{eq:Fprime-Ric} at $\lambda=0$, all terms
involving $\ddot g_0$, the variation of the inverse metric in the
contraction, and the variation of $dV_{g_\lambda}$ vanish because
they are multiplied by $\text{Ric}(\delta)$.  Thus
\begin{equation}
\label{eq:Fsecond-Ric-linearization}
2F''(0)
=
-
\int_{\Omega_0}
\left\langle
h,D\text{Ric}_\delta(h)
\right\rangle_\delta
\,dV_\delta.
\end{equation}
The linearization of the Ricci tensor at $\delta$ is
\begin{equation}
\label{eq:Ricci-linearization-flat}
\bigl(D\text{Ric}_\delta(h)\bigr)_{ij}
=
\frac12
\left(
\partial^k\partial_i h_{kj}
+
\partial^k\partial_j h_{ki}
-
\Delta_\delta h_{ij}
-
\partial_i\partial_j\tr_\delta h
\right).
\end{equation}
Since
\[
\partial^kh_{ki}=0,
\qquad
\tr_\delta h=0,
\]
we obtain
\begin{equation}
\label{eq:Ricci-linearization-TT}
D\text{Ric}_\delta(h)
=
-\frac12\Delta_\delta h.
\end{equation}
Consequently,
\begin{align}
\int_{\Omega_0}
\left\langle h,D\text{Ric}_\delta(h)\right\rangle_\delta\,dV_\delta
&=
-\frac12
\int_{\Omega_0}
h^{ij}\Delta_\delta h_{ij}\,dV_\delta
\nonumber\\
&=
\frac12
\int_{\Omega_0}
|\nabla^\delta h|_\delta^2\,dV_\delta
-
\frac12
\int_{\Sigma_0}
h^{ij}
(\nabla^\delta_{\nu_0}h)_{ij}
\,dA_{\gamma_0}.
\label{eq:hRicprime-integration}
\end{align}
It remains to compute the boundary contraction.  Fix
$p\in\Sigma_0$ and choose a $\delta$--orthonormal frame
\[
\{e_1,e_2,\nu_0\}
\]
near $p$, with $e_1,e_2$ tangent to $\Sigma_0$, such that
\[
\nabla^{\gamma_0}_{e_b}e_a=0,
\qquad
\nabla^\delta_{\nu_0}e_a=0
\qquad\text{at }p.
\]
Write
\[
Z_a:=h(\nu_0,e_a).
\]
The boundary condition $h^T=0$ gives
\begin{equation}
\label{eq:h-ab-zero}
h(e_a,e_b)=0
\qquad\text{on }\Sigma_0.
\end{equation}
Since $h$ is trace-free,
\[
0
=
\tr_\delta h
=
h(\nu_0,\nu_0)
+
\sum_{a=1}^2h(e_a,e_a),
\]
and hence
\begin{equation}
\label{eq:h-nunu-zero}
h(\nu_0,\nu_0)=0
\qquad\text{on }\Sigma_0.
\end{equation}
Thus only the mixed components $h(\nu_0,e_a)=Z_a$ can be nonzero on
$\Sigma_0$. The divergence-free condition in the $e_a$ direction gives
\begin{equation}
\label{eq:div-h-a}
0
=
(\div_\delta h)(e_a)
=
(\nabla_{\nu_0}h)(\nu_0,e_a)
+
\sum_{b=1}^2
(\nabla_{e_b}h)(e_b,e_a).
\end{equation}
With our convention
\[
A_0(Y,Z)
=
\delta(\nabla^\delta_Y\nu_0,Z),
\]
the connection equations are
\begin{equation}
\label{eq:Gauss-boundary-frame}
\nabla^\delta_{e_b}\nu_0
=
A_{0\,bc}e_c,
\qquad
\nabla^\delta_{e_b}e_a
=
\nabla^{\gamma_0}_{e_b}e_a
-
A_{0\,ba}\nu_0.
\end{equation}
Since $h(e_b,e_a)$ vanishes identically along $\Sigma_0$,
\[
e_b\!\left(h(e_b,e_a)\right)=0
\qquad\text{at }p.
\]
Therefore, at $p$,
\begin{align}
(\nabla_{e_b}h)(e_b,e_a)
&=
e_b\!\left(h(e_b,e_a)\right)
-
h(\nabla_{e_b}e_b,e_a)
-
h(e_b,\nabla_{e_b}e_a)
\nonumber\\
&=
A_{0\,bb}h(\nu_0,e_a)
+
A_{0\,ba}h(e_b,\nu_0)
\nonumber\\
&=
A_{0\,bb}Z_a+A_{0\,ba}Z_b.
\label{eq:tangential-derivative-h}
\end{align}
Summing in $b$,
\begin{equation}
\label{eq:tangential-div-h}
\sum_{b=1}^2
(\nabla_{e_b}h)(e_b,e_a)
=
H_0Z_a+(A_0Z_h)_a.
\end{equation}
Substitution into \eqref{eq:div-h-a} yields
\begin{equation}
\label{eq:normal-derivative-mixed-h}
(\nabla_{\nu_0}h)(\nu_0,e_a)
=
-H_0Z_a-(A_0Z_h)_a.
\end{equation}
Using \eqref{eq:h-ab-zero} and \eqref{eq:h-nunu-zero},
\begin{align}
h^{ij}(\nabla_{\nu_0}h)_{ij}
&=
2\sum_{a=1}^2
h(\nu_0,e_a)
(\nabla_{\nu_0}h)(\nu_0,e_a)
\nonumber\\
&=
-2H_0\sum_{a=1}^2Z_a^2
-
2\sum_{a,b=1}^2A_{0\,ab}Z_aZ_b
\nonumber\\
&=
-2H_0|Z_h|_{\gamma_0}^2
-
2A_0(Z_h,Z_h).
\label{eq:boundary-contraction-final}
\end{align}
Hence \eqref{eq:hRicprime-integration} becomes
\begin{equation}
\label{eq:hRicprime-final}
\begin{split}
\int_{\Omega_0}
\left\langle h,D\text{Ric}_\delta(h)\right\rangle_\delta\,dV_\delta
={}&
\frac12
\int_{\Omega_0}
|\nabla^\delta h|_\delta^2\,dV_\delta
\\
&+
\int_{\Sigma_0}
\left(
A_0(Z_h,Z_h)
+
H_0|Z_h|_{\gamma_0}^2
\right)dA_{\gamma_0}.
\end{split}
\end{equation}
Combining \eqref{eq:Fsecond-Ric-linearization} and
\eqref{eq:hRicprime-final},
\begin{equation}
\label{eq:Fsecond-final}
\begin{split}
F''(0)
={}&
-\frac14
\int_{\Omega_0}
|\nabla^\delta h|_\delta^2\,dV_\delta
\\
&-
\frac12
\int_{\Sigma_0}
\left(
A_0(Z_h,Z_h)
+
H_0|Z_h|_{\gamma_0}^2
\right)dA_{\gamma_0}.
\end{split}
\end{equation}
Using \eqref{eq:BY-vs-F},
\[
8\pi m_{\rm BY}''(0)=-F''(0),
\]
and therefore
\[
\begin{split}
8\pi m_{\rm BY}''(0)
={}&
\frac14
\int_{\Omega_0}
|\nabla^\delta h|_\delta^2\,dV_\delta
\\
&+
\frac12
\int_{\Sigma_0}
\left(
A_0(Z_h,Z_h)
+
H_0|Z_h|_{\gamma_0}^2
\right)dA_{\gamma_0},
\end{split}
\]
which proves \eqref{eq:BY-TT-second}.
Finally, strict convexity (note that $\Sigma_{0}$ has strictly positive Gauss curvature and $H_{0}>0$ and therefore $A_{0}$ is positive definite) gives
\[
A_0>0,
\qquad
H_0>0,
\]
so both terms in \eqref{eq:BY-TT-second} are nonnegative.  If
$m_{\rm BY}''(0)=0$, then
\[
\nabla^\delta h\equiv0,
\qquad
Z_h\equiv0.
\]
Thus $h$ is parallel.  On $\Sigma_0$,
\[
h^T=0,
\qquad
h(\nu_0,e_a)=0,
\qquad
h(\nu_0,\nu_0)=0,
\]
the last identity following from $\tr_\delta h=0$.  Hence
$h=0$ at every boundary point.  Since $\Omega_0$ is connected and
$\nabla^\delta h=0$, it follows that $h\equiv0$ on $\Omega_0$,
contrary to the hypothesis.  Therefore
\[
m_{\rm BY}''(0)>0.
\]
\end{proof}

\begin{corollary}[Quadratic expansion]
\label{cor:BY-quadratic-expansion}

Under the assumptions of Lemma~\ref{lem:TT-BY-Hessian}, one has
\begin{equation}
\label{eq:BY-expansion}
\begin{split}
m_{\rm BY}(g_\lambda)
={}&
\frac{\lambda^2}{64\pi}
\int_{\Omega_0}
|\nabla^\delta h|_\delta^2\,dV_\delta
\\
&+
\frac{\lambda^2}{32\pi}
\int_{\Sigma_0}
\left(
A_0(Z_h,Z_h)
+
H_0|Z_h|_{\gamma_0}^2
\right)dA_{\gamma_0}
+
o(\lambda^2).
\end{split}
\end{equation}
In particular,
\begin{equation}
\label{eq:BY-small-positive}
m_{\rm BY}(g_\lambda)>0
\end{equation}
for all sufficiently small $\lambda\neq0$.

\end{corollary}

\begin{proof}
Since
\[
g_0=\delta,
\]
the Brown--York mass vanishes at $\lambda=0$:
\[
m_{\rm BY}(g_0)=0.
\]
By Lemma~\ref{lem:TT-BY-Hessian},
\[
m_{\rm BY}'(0)=0,
\]
and
\[
8\pi m_{\rm BY}''(0)
=
\frac14
\int_{\Omega_0}|\nabla^\delta h|_\delta^2\,dV_\delta
+
\frac12
\int_{\Sigma_0}
\left(
A_0(Z_h,Z_h)
+
H_0|Z_h|_{\gamma_0}^2
\right)dA_{\gamma_0}.
\]
Taylor expansion at $\lambda=0$ therefore gives
\[
m_{\rm BY}(g_\lambda)
=
\frac{\lambda^2}{2}m_{\rm BY}''(0)
+
o(\lambda^2),
\]
which is precisely \eqref{eq:BY-expansion}.
Set
\[
\mathcal Q(h)
:=
\frac{1}{64\pi}
\int_{\Omega_0}|\nabla^\delta h|_\delta^2\,dV_\delta
+
\frac{1}{32\pi}
\int_{\Sigma_0}
\left(
A_0(Z_h,Z_h)
+
H_0|Z_h|_{\gamma_0}^2
\right)dA_{\gamma_0}.
\]
Lemma~\ref{lem:TT-BY-Hessian} gives
\[
\mathcal Q(h)>0
\]
for every nonzero $h\in \mathcal{T}^{k,\alpha}$.  Hence
\[
m_{\rm BY}(g_\lambda)
=
\lambda^2\bigl(\mathcal Q(h)+o(1)\bigr).
\]
Thus, after decreasing $\varepsilon>0$ if necessary,
\[
\mathcal Q(h)+o(1)\geq \frac12\mathcal Q(h)>0
\]
whenever $0<|\lambda|<\varepsilon$, proving
\eqref{eq:BY-small-positive}.
\end{proof}

\subsection{Proof of the main theorem}
\label{sec:proof-main-theorem}
\noindent In this section, we finish the proof of the main theorem. Let $\tau$ be a minimizing admissible Wang-Yau time function for the
fixed physical surface $\Sigma$, and let
\[
(\widehat\Omega,\widehat g,X)
\]
be the corresponding Jang data.  Let $U>0$ solve
\[
-8\Delta_{\widehat g}U+R_{\widehat g}U=0
\quad\text{in }\widehat\Omega,
\qquad
U=1
\quad\text{on }\partial\widehat\Omega,
\]
and set
\[
g_{\rm sf}:=U^4\widehat g.
\]
By the scalar-flat reduction established above,
\begin{equation}
\label{eq:main-WY-BY-comparison}
\begin{split}
E_{\rm WY}(\Sigma,\tau)
\geq{}&
m_{\rm BY}
(\partial\widehat\Omega;\widehat\Omega,g_{\rm sf})
\\
&+
\frac{3}{8\pi}
\int_{\widehat\Omega}
|\widehat\nabla U|_{\widehat g}^2\,dV_{\widehat g}
+
\frac{1}{8\pi}
\int_{\widehat\Omega}
|\widehat\nabla U+UX|_{\widehat g}^2\,dV_{\widehat g}.
\end{split}
\end{equation}
In particular,
\begin{equation}
\label{eq:main-WY-ge-BY}
E_{\rm WY}(\Sigma,\tau)
\geq
m_{\rm BY}
(\partial\widehat\Omega;\widehat\Omega,g_{\rm sf}).
\end{equation}
By the TT-generated scalar-flat construction established above, after a
boundary-preserving identification
\[
\Phi:\Omega_{0}\longrightarrow\widehat\Omega,
\]
the scalar-flat conformal Jang metric can be represented in the form
\begin{equation}
\label{eq:main-proof-TT-representation}
\Phi^{*}g_{\rm sf}
=
G(\lambda h)
=
v_{\lambda}^{4}(\delta+\lambda h),
\end{equation}
where
\[
0\neq h\in C^{k,\alpha}
(\Omega_{0};S^{2}T^{*}\Omega_{0}),
\]
satisfies
\[
\operatorname{div}_{\delta}h=0,
\qquad
\operatorname{tr}_{\delta}h=0,
\qquad
h^{T}|_{\Sigma_{0}}=0,
\]
and $v_{\lambda}>0$ is the solution of
\[
-8\Delta_{\delta+\lambda h}v_{\lambda}
+
R_{\delta+\lambda h}v_{\lambda}=0,
\qquad
v_{\lambda}|_{\Sigma_{0}}=1.
\]  Since the Brown--York mass
is invariant under boundary-preserving diffeomorphisms,
\begin{equation}
\label{eq:BY-diffeo-invariance}
m_{\rm BY}
(\partial\widehat\Omega;\widehat\Omega,g_{\rm sf})
=
m_{\rm BY}
(\Sigma_0;\Omega_0,\mathfrak G(\lambda h)).
\end{equation}
Corollary~\ref{cor:BY-quadratic-expansion} yields
\begin{equation}
\label{eq:main-BY-expansion}
\begin{split}
m_{\rm BY}
(\Sigma_0;\Omega_0,\mathfrak G(\lambda h))
={}&
\frac{\lambda^2}{64\pi}
\int_{\Omega_0}
|\nabla^\delta h|_\delta^2\,dV_\delta
\\
&+
\frac{\lambda^2}{32\pi}
\int_{\Sigma_0}
\left(
A_0(Z_h,Z_h)
+
H_0|Z_h|_{\gamma_0}^2
\right)dA_{\gamma_0}
+
o(\lambda^2).
\end{split}
\end{equation}
The quadratic coefficient is strictly positive by
Lemma~\ref{lem:TT-BY-Hessian}.  Therefore
\begin{equation}
\label{eq:main-BY-positive}
m_{\rm BY}
(\partial\widehat\Omega;\widehat\Omega,g_{\rm sf})
>0
\end{equation}
for every sufficiently small nonzero $\lambda$. Combining \eqref{eq:main-WY-ge-BY} and
\eqref{eq:main-BY-positive} gives
\[
E_{\rm WY}(\Sigma,\tau)>0.
\]
Finally, since $\tau$ is minimizing,
\[
M_{\rm WY}(\Sigma)
=
\inf_{\tau'\in\mathcal A(\Sigma)}
E_{\rm WY}(\Sigma,\tau')
=
E_{\rm WY}(\Sigma,\tau),
\]
and hence
\begin{equation}
\label{eq:main-WY-positive}
M_{\rm WY}(\Sigma)>0.
\end{equation}
This completes the proof of the main theorem.


\begin{thebibliography}{}

\bibitem{alaee2023quasi}
A. Alaee, M. Khuri, S-T. Yau,
A Quasi-Local Mass, arXiv:2309.02770, 2023.


\bibitem{atiyah1975spectral1}
M.F. Atiyah, V.K. Patodi, I.M. Singer,
Spectral asymmetry and Riemannian geometry. I,
\textit{Mathematical Proceedings of the Cambridge Philosophical Society},
vol. 77, 43-69, 1975.

\bibitem{atiyah1975spectral2}
M.F. Atiyah, V.K. Patodi, I.M. Singer,
Spectral asymmetry and Riemannian geometry. I,
\textit{Mathematical Proceedings of the Cambridge Philosophical Society},
vol. 78, 405-432, 1975.

\bibitem{atiyah1975spectral3}
M.F. Atiyah, V.K. Patodi, I.M. Singer,
Spectral asymmetry and Riemannian geometry. I,
\textit{Mathematical Proceedings of the Cambridge Philosophical Society},
vol. 79, 71-99, 1976.


\bibitem{bar1999zero}
C. B{\"a}r,
Zero sets of solutions to semilinear elliptic systems of the first order,
\textit{Inventiones mathematicae}, vol. 138, 183-202, 1999.

\bibitem{bar}
C. B{\"a}r, W. Ballmann,
Guide to elliptic boundary value problems for Dirac-type operators,
2016, Springer.

\bibitem{Bar}
C. B\"ar,
Localization and semibounded energy---a weak unique continuation theorem,
\textit{Journal of Geometry and Physics}, vol. 34, 155-161, 2000.


\bibitem{bartnik1993quasi}
R. Bartnik,
Quasi-spherical metrics and prescribed scalar curvature,
\textit{Journal of differential geometry}, vol. 37, 21-71, 1993.


\bibitem{bondi1962gravitational}
H. Bondi, M.G.J Van der Burg, AWK, Metzner,
Gravitational waves in general relativity, VII. Waves from axi-symmetric isolated system,
\textit{Proceedings of the Royal Society of London. Series A. Mathematical and Physical Sciences},
vol. 269, 21-52, 1962.


\bibitem{brown1992quasilocal}
J.D. Brown, J.W. York,
Quasilocal energy in general relativity,
\textit{Mathematical aspects of classical field theory},
vol. 132, 129-142, 1992.

\bibitem{brown1993quasilocal}
J.D. Brown, J.W. York,
Quasilocal energy and conserved charges derived from the gravitational action,
\textit{Physical Review D}, vol. 47, 1407, 1993.


\bibitem{chen2011evaluating}
N.P. Chen, M.T. Wang, S.T. Yau,
Evaluating quasilocal energy and solving optimal embedding equation at null infinity,
\textit{Communications in mathematical physics},
vol. 308, 845-863, 2011.

\bibitem{chen2015conserved}
N.P. Chen, M.T. Wang, S.T. Yau,
Conserved quantities in general relativity: from the quasi-local level to spatial infinity,
\textit{Communications in Mathematical Physics},
vol. 338, 31-80, 2015, Springer.

\bibitem{chen2018evaluating}
N.P. Chen, M.T. Wang, S.T. Yau,
Evaluating small sphere limit of the Wang-Yau quasi-local energy,
\textit{Communications in Mathematical Physics},
vol. 357, 731-774, 2018, Springer.


\bibitem{chrusciel2019hyperbolic}
P.T. Chru{\'s}ciel, E. Delay,
The hyperbolic positive energy theorem,
arXiv:1901.05263, 2019.

\bibitem{chrusciel2003mass}
P.T. Chru{\'s}ciel, M. Herzlich,
The mass of asymptotically hyperbolic Riemannian manifolds,
\textit{Pacific journal of mathematics}, 231-264, 2003.


\bibitem{ginoux2009dirac}
N. Ginoux,
The dirac spectrum, vol. 1976, 2009.


\bibitem{jang1978positivity}
P.S. Jang,
On the positivity of energy in general relativity,
\textit{Journal of Mathematical Physics},
vol. 19, 1152-1155, 1978.


\bibitem{johnson1975bag}
K. Johnson,
The MIT bag model,
Acta Phys. Pol. B, vol. 6, 1975.


\bibitem{liu2003positivity}
C.M. Liu, S.T. Yau,
Positivity of quasilocal mass,
\textit{Physical review letters},
vol. 90, 231102, 2003.

\bibitem{liu2006positivity}
C.M. Liu, S.T. Yau,
Positivity of quasi-local mass II,
\textit{Journal of the American Mathematical Society},
vol. 19, 181-204, 2006.


\bibitem{lott}
J. Lott,
A spinorial quasilocal mass,
\textit{J. Math. Phys.}, vol. 64, 2023.


\bibitem{pyau}
P. Mondal, S-T Yau,
Quasi-local masses in General relativity and their positivity: Spinor approach,
arXiv:2401.13909, 2024.





\bibitem{pyau2} P. Mondal, S-T Yau, Aspects of quasilocal energy for gravity coupled to gauge fields, Physical Rev. D., Vol. 105, 104068, 2022. 






\bibitem{montiel2022compact}
S. Montiel,
Compact approach to the positivity of Brown-York mass,
arXiv:2209.07762, 2022.


\bibitem{murchadha2004comment}
N \'O, L.B. Szabados, K.P. Tod,
Comment on ``Positivity of quasilocal mass",
\textit{Physical review letters},
vol. 92, 259001, 2004.


\bibitem{parker1982witten}
T. Parker, C.H. Taubes,
On Witten's proof of the positive energy theorem,
\textit{Communications in Mathematical Physics},
vol. 84, 223-238, 1982.


\bibitem{penrose1965gravitational}
R. Penrose,
Gravitational collapse and space-time singularities,
\textit{Physical Review Letters},
vol. 14, 57, 1965.

\bibitem{penrose1982some}
R. Penrose,
Some unsolved problems in classical general relativity,
\textit{Annals of Mathematics Studies},
vol. 102, 631-668, 1982.


\bibitem{pogorelov1952regularity}
A.V. Pogorelov,
Regularity of a convex surface with given Gaussian curvature,
\textit{Matematicheskii Sbornik},
vol. 73, 88-103, 1952.


\bibitem{schoen1979proof}
R. Schoen, S.T. Yau,
On the proof of the positive mass conjecture in general relativity,
\textit{Communications in Mathematical Physics},
vol. 65, 45-76, 1979.

\bibitem{schoen1981proof}
R. Schoen, S.T. Yau,
Proof of the positive mass theorem. II,
\textit{ommunications in Mathematical Physics},
vol. 79, 231-260, 1981.

\bibitem{schoen1982proof}
R. Schoen, S.T. Yau,
Proof that the Bondi mass is positive,
\textit{Physical Review Letters},
vol. 48, 369, 1982.

\bibitem{schoen1}
R. Schoen,
Talk at a conference on Geometric Analysis and General Relativity,
University of Tokyo, November 2019.


\bibitem{shi2002positive}
Y. Shi, L-F. Tam,
Positive mass theorem and the boundary behaviors of compact manifolds with nonnegative scalar curvature,
\textit{Journal of Differential Geometry},
vol. 62, 79-125, 2002.


\bibitem{yau}
M.T. Wang, S.T. Yau,
Isometric embeddings into the Minkowski space and new quasi-local mass,
\textit{Communications in Mathematical Physics},
vol. 288, 919-942, 2009, Springer.

\bibitem{yau1}
M.T. Wang, S.T. Yau,
Quasilocal mass in general relativity,
\textit{Physical review letters},
vol. 102, 021101, 2009.

\bibitem{wang2010limit}
M.T. Wang, S.T. Yau,
Limit of quasilocal mass at spatial infinity,
\textit{Communications in Mathematical Physics},
vol. 296, 271-283, 2010, Springer.


\bibitem{witten1981new}
E. Witten,
A new proof of the positive energy theorem,
\textit{Communications in Mathematical Physics},
vol. 80, 381-402, 1981.





\end{thebibliography}
\end{document}